\documentclass[10pt,reqno]{amsart}
\usepackage{pdfsync}
\usepackage{amssymb,amscd}
\usepackage{color}   

\makeatletter

\@addtoreset{equation}{section}
\makeatother
                          
\newtheorem{theorem}{Theorem}[section]                    
\newtheorem{proposition}[theorem]{Proposition}            
\newtheorem{corollary}[theorem]{Corollary}                
\newtheorem{lemma}[theorem]{Lemma}
\newtheorem{remark}[theorem]{Remark}

\newtheorem{definition}[theorem]{Definition}

\begin{document}
\title[Orbital Free Entropy, revisited]{Orbital Free Entropy, revisited}
\author[Y.~Ueda]{Yoshimichi Ueda}
\address{
Graduate School of Mathematics, 
Kyushu University, 
Fukuoka, 819-0395, Japan
}
\email{ueda@math.kyushu-u.ac.jp}
\thanks{Supported by Grant-in-Aid for Scientific Research (C) 24540214.}
\dedicatory{Dedicated to Professor Yasuo Watatani on the occasion of his 60th birthday}
\begin{abstract} We give another definition of orbital free entropy introduced by Hiai, Miyamoto and us, which does not need the hyperfiniteness assumption for each given random multi-variable. The present definition is somehow related to one of its several recent approaches due to Biane and Dabrowski, but can be shown to agree with the original definition completely and is much closer to the original approach.  
\end{abstract} 
\maketitle

\allowdisplaybreaks{

\section{Introduction}

Let $(\mathcal{M},\tau)$ be a tracial non-commutative $W^*$-probability space throughout the present paper. We call any {\it finite} tuple $\mathbf{X} = (X_1,\dots,X_r)$ of self-adjoint elements in $\mathcal{M}$ a (non-commutative) random multi-variable. When the tuple $\mathbf{X}$ consists of just one element, i.e., $\mathbf{X} = (X)$, the random multi-variable should be regarded as nothing less than a (non-commutative) random variable $X$. We want to measure how far the joint distribution of given random multi-variables $\mathbf{X}_1,\dots,\mathbf{X}_n$ is from that of freely independent copies of those, i.e., a freely independent family $\mathbf{Y}_1,\dots,\mathbf{Y}_n$ such that each $\mathbf{Y}_i$ has the same joint distribution as that of $\mathbf{X}_i$. The first one of such quantities, called the free mutual information, $i^*(W^*(\mathbf{X}_1)\,;\,\dots\,;\,W^*(\mathbf{X}_n))$, was introduced by Voiculescu \cite{Voiculescu:AdvMath99} in the so-called microstate-free approach to free entropy, where $W^*(\mathbf{X}_i)$ denotes the von Neumann subalgebra of $\mathcal{M}$ generated by (the random variables in) $\mathbf{X}_i$. Almost 10 years later then, the second, called the orbital free entropy, $\chi_\mathrm{orb}(\mathbf{X}_1,\dots,\mathbf{X}_n)$, was introduced by Hiai, Miyamoto and us \cite{HiaiMiyamotoUeda:IJM09} in the so-called microstate approach. Those two quantities possess many properties in common, and hence we may conjecture that those essentially define the same quantity under a suitable set of assumptions. In the direction a heuristic argument supporting the conjecture was given in \cite{HiaiUeda:AIHP09} when $n=2$ and each of those $\mathbf{X}_1,\mathbf{X}_2$ consists of just one projection, and soon after then it was confirmed in a different way with the help of M.~Izumi under a certain set of assumptions on those projections. Hence we have thought that the conjecture is plausible (but, of course, very difficult like the question of whether or not $\chi = \chi^*$ holds).  On the other hand, those quantities $i^*$ and $\chi_\mathrm{orb}$ have individual defects: Nobody knows whether or not 
$$
\chi^*(X,Y) = -i^*(W^*(X)\,;\,W^*(Y)) + \chi^*(X) + \chi^*(Y)
$$
holds; $\chi_\mathrm{orb}$ is unfortunately defined only when each $W^*(\mathbf{X}_i)$ is hyperfinite, though it satisfies that 
\begin{equation}\label{Eq0-1}
\chi(X_1,\dots,X_n) = \chi_\mathrm{orb}(X_1,\dots,X_n) + \sum_{i=1}^n \chi(X_i). 
\end{equation}
Recently Biane and Dabrowski \cite{BianeDabrowski:AdvMath13} proposed and developed various approaches of $\chi_\mathrm{orb}$ that does no longer need the hyperfiniteness assumption on the $W^*(\mathbf{X}_i)$'s, though they identified 
those with original $\chi_\mathrm{orb}$ only when $W^*(\mathbf{X}_1\sqcup\cdots\sqcup\mathbf{X}_n)$ is a factor (and, of course, each $W^*(\mathbf{X}_i)$ is hyperfinite). Motivated by their work we give another definition of original $\chi_\mathrm{orb}$ that works without the hyperfiniteness assumption, and prove almost all the expected properties of $\chi_\mathrm{orb}$ including the following: The orbital free entropy $\chi_\mathrm{orb}(\mathbf{X}_1,\dots,\mathbf{X}_n)$ depends only on the $W^*(\mathbf{X}_i)$'s, and moreover $\chi_\mathrm{orb}(\mathbf{X}_1,\dots,\mathbf{X}_n) = 0$ if and only if the $\mathbf{X}_i$'s have f.d.a.~and are freely independent. However, we do not know whether or not the counterpart of \eqref{Eq0-1} holds in general, and at the moment we can prove only the inequality  
\begin{equation}\label{Eq0-2}
\chi(\mathbf{X}_1\sqcup\cdots\sqcup\mathbf{X}_n) \leq \chi_\mathrm{orb}(\mathbf{X}_1,\dots,\mathbf{X}_n) + \sum_{i=1}^n\chi(\mathbf{X}_i). 
\end{equation}
It is apparently our task for the future to answer the question of whether or not equality in \eqref{Eq0-2} holds in general. Here we would like to emphasize that many arguments given in the present paper originate in \cite{HiaiMiyamotoUeda:IJM09}, and the technical ingredients here, similarly to \cite{HiaiMiyamotoUeda:IJM09}, are only three non-trivial previous results, two of which are due to Voiculescu and the other is due to Jung, summarized in \cite[\S1]{HiaiMiyamotoUeda:IJM09}. 

With the new definition of $\chi_\mathrm{orb}$ one can generalize orbital free entropy dimension $\delta_{0,\mathrm{orb}}$ to arbitrary random multi-variables, and we see that Jung's covering/packing approach still works well without the hyperfiniteness assumption. However, the formula \cite[Theorem 5.6]{HiaiMiyamotoUeda:IJM09}, analogous to \eqref{Eq0-1}, for $\delta_{0,\mathrm{orb}}$ and free entropy dimension $\delta_0$ is out of reach without the hyperfiniteness assumption. Moreover, we do not know at the moment whether or not even the analogous inequality to \eqref{Eq0-2} holds for $\delta_{0,\mathrm{orb}}$ and $\delta_0$ in general. Answering this should also be our task for the future.

In the present paper, we denote all the $N\times N$ matrices by $M_N$, all the self-adjoint $N\times N$ matrices by $M_N^\mathrm{sa}$, and tuples of such matrices by letters $\mathbf{A},\mathbf{B}$, etc., called multi-matrices, while letters $\mathbf{X},\mathbf{Y}$, etc.~stand for random multi-variables as before. The normalized trace on the $N\times N$ matrices $M_N$ is denoted by $\mathrm{tr}_N$. The Lebesgue measure on $M_N^\mathrm{sa} \cong \mathbb{R}^{N^2}$ is denoted by $\Lambda_N$. The operator norm is written as $\Vert-\Vert_\infty$, while the $p$-norm ($1 \leq p < \infty$) determined by $\varphi = \tau$ or $\mathrm{tr}_N$ is defined to be $\Vert x \Vert_{p,\varphi} := \varphi(|x|^p)^{1/p}$. Write $(M_N^\mathrm{sa})_R := \{ A \in M_N^\mathrm{sa}\,|\,\Vert A\Vert_\infty \leq R \}$. The $N\times N$ unitary group and the $N\times N$ special unitary group are denoted by $\mathrm{U}(N)$ and $\mathrm{SU}(N)$, respectively, and their unique Haar probability measures by $\gamma_{\mathrm{U}(N)}$ and $\gamma_{\mathrm{SU}(N)}$, respectively. The regular Borel probability measures on a locally compact Hausdorff space $\mathcal{X}$ is denoted by $\mathcal{P}(\mathcal{X})$.

\section{Orbital Free Entropy $\chi_\mathrm{orb}(\mathbf{X}_1,\dots,\mathbf{X}_n)$} 

Let $\mathbf{X}_i = (X_{i1},\dots,X_{ir(i)})$, $1 \leq i \leq n$, be arbitrary random multi-variables in $(\mathcal{M},\tau)$ ({\it n.b.}~$r(i)$ is finite as mentioned in the introduction), and $R > 0$ be given possibly with $R=\infty$. Let $m \in \mathbb{N}$ and $\delta > 0$ be given. The set of matricial microstates $\Gamma_R(\mathbf{X}_1\sqcup\cdots\sqcup\mathbf{X}_n\,;\,N,m,\delta)$ is defined to be all the $n$ tuples of $N\times N$ multi-matrices $\mathbf{A}_i = (A_{i1},\dots,A_{ir(i)}) \in ((M_N^\mathrm{sa})_R)^{r(i)}$, $1 \leq i \leq n$, such that 
\begin{equation*}
\left|\mathrm{tr}_N(A_{i_1 j_1} \cdots A_{i_l j_l}) - \tau(X_{i_1 j_1} \cdots X_{i_l j_l})\right| < \delta
\end{equation*}
whenever $1 \leq i_k \leq n$, $1 \leq j_k \leq r(i_k)$, $1 \leq k \leq l$ and $1 \leq l \leq m$.  Similarly the set of matricial microstates $\Gamma_R(\mathbf{X}_i\,;\,N,m,\delta)$ are defined for each random multi-variable $\mathbf{X}_i$, $1 \leq i \leq n$. Consider the map $\Phi_N : \mathrm{U}(N)^n \times \left(\prod_{i=1}^n \big(M_N^\mathrm{sa}\big)^{r(i)}\right) \longrightarrow \prod_{i=1}^n \big(M_N^\mathrm{sa}\big)^{r(i)}$ defined by 
\begin{equation*}
\Phi_N((U_i)_{i=1}^n,(\mathbf{A}_i)_{i=1}^n) := (U_i\mathbf{A}_i U_i^*)_{i=1}^n
\end{equation*}
with $U_i\mathbf{A}_i U_i^* := (U_i A_{i1}U_i^*,\dots,U_i A_{ir(i)} U_i^*)$. Here is a new definition of the orbital free entropy $\chi_\mathrm{orb}(\mathbf{X}_1,\dots,\mathbf{X}_n)$ \cite{HiaiMiyamotoUeda:IJM09}. 

\begin{definition}\label{D-2-1} For each $N, m \in \mathbb{N}$, $\delta > 0$ and $\mu \in \mathcal{P}(\prod_{i=1}^n(M_N^\mathrm{sa})^{r(i)})$ we define 
\begin{align*} 
\chi_{\mathrm{orb},R}(\mathbf{X}_1,\dots,\mathbf{X}_n\,;\,N,m,\delta\,;\,\mu)
&:= \log\big((\gamma_{\mathrm{U}(N)}^{\otimes n}\otimes\mu)(\Phi_N^{-1}(\Gamma_R(\mathbf{X}_1\sqcup\cdots\sqcup\mathbf{X}_n\,;\,N,m,\delta)))\big), \\
\chi_{\mathrm{orb},R}(\mathbf{X}_1,\dots,\mathbf{X}_n\,;\,N,m,\delta) 
&:=
\sup_\mu\chi_{\mathrm{orb},R}(\mathbf{X}_1,\dots,\mathbf{X}_n\,;\,N,m,\delta\,;\,\mu), \\
\chi_{\mathrm{orb},R}(\mathbf{X}_1,\dots,\mathbf{X}_n\,;\,m,\delta) 
&:= 
\limsup_{N\to\infty}\frac{1}{N^2}\chi_{\mathrm{orb},R}(\mathbf{X}_1,\dots,\mathbf{X}_n\,;\,N,m,\delta), \\
\chi_{\mathrm{orb},R}(\mathbf{X}_1,\dots,\mathbf{X}_n) 
&:= 
\lim_{\substack{m\to\infty \\ \delta \searrow 0}}\,\chi_{\mathrm{orb},R}(\mathbf{X}_1,\dots,\mathbf{X}_n\,;\,m,\delta) \\
&\,= 
\inf_{\substack{m\in\mathbb{N} \\ \delta > 0}}\,\chi_{\mathrm{orb},R}(\mathbf{X}_1,\dots,\mathbf{X}_n\,;\,m,\delta), \\
\chi_\mathrm{orb}(\mathbf{X}_1,\dots,\mathbf{X}_n) 
&:= \sup_{0 < R < \infty}\chi_{\mathrm{orb},R}(\mathbf{X}_1,\dots,\mathbf{X}_n)
\end{align*} 
with $\log0:= -\infty$. 
\end{definition} 

It will be seen in Proposition \ref{P-2-4} that the above quantity turns out to agree with the quantity appeared in \cite[Theorem 7.3 (7)]{BianeDabrowski:AdvMath13}. Although the reformulation of $\chi_\mathrm{orb}$ given there is probably easier to use in many actual computations of $\chi_\mathrm{orb}$, the above definition is directly related to Voiculescu's microstate free entropy $\chi$ and has a bit advantage in some arguments, see Proposition \ref{P-2-8} and Remark \ref{R-2-9} (also Theorem \ref{T-2-6} (5)). Originally the present work started with examining whether or not the fruitful idea of `random microstates' due to Biane and Dabrowski \cite{BianeDabrowski:AdvMath13} enables us to develop a multi-matrix counterpart of the matrix diagonalization (c.f.~\cite[Eq.(1.1)]{HiaiMiyamotoUeda:IJM09}) which plays a particularly important r\^{o}le in our previous analysis. In conclusion, we noticed that the simple fact given as Lemma \ref{L-2-2} below should play a key r\^{o}le in any possible approach, and immediately observed what we explain in Remark \ref{R-2-9}. We believe that the present approach that we are developing can be a basis for future studies of orbital free entropy.     

\medskip
Let us first prove that the new definition of $\chi_\mathrm{orb}$ perfectly agrees with the previous one when each $W^*(\mathbf{X}_i)$ is hyperfinite. Recall the necessary notations that appear in the previous definition of $\chi_\mathrm{orb}$. For given multi-matrices $\mathbf{A}_i = (A_{ij})_{j=1}^{r(i)} \in (M_N^\mathrm{sa})^{r(i)}$, $1 \leq i \leq n$, the set of orbital microstates $\Gamma_\mathrm{orb}(\mathbf{X}_1,\dots,\mathbf{X}_n:(\mathbf{A}_i)_{i=1}^n\,;\,N,m,\delta)$ is defined to be all $(U_i)_{i=1}^n \in \mathrm{U}(N)^n$ such that $(U_i\mathbf{A}_i U_i^*)_{i=1}^n \in \Gamma_\infty(\mathbf{X}_1\sqcup\cdots\sqcup\mathbf{X}_n\,;\,N,m,\delta)$. It is trivial that if some $\mathbf{A}_i$ sits in $((M_N^\mathrm{sa})_R)^{r(i)}\setminus\Gamma_R(\mathbf{X}_i\,;\,N,m,\delta)$, then $\Gamma_\mathrm{orb}(\mathbf{X}_1,\dots,\mathbf{X}_n:(\mathbf{A}_i)_{i=1}^n\,;\,N,m,\delta) = \emptyset$. The next simple fact explains a relation between matricial microstate spaces and orbital microstate spaces; thus we display it as a separate lemma. 

\begin{lemma}\label{L-2-2} For every $R>0$ possibly with $R = \infty$ the function 
$$
(\mathbf{A}_i)_{i=1}^n \in \prod_{1 \leq i \leq n} ((M_N^\mathrm{sa})_R)^{r(i)} \mapsto \gamma_{\mathrm{U}(N)}^{\otimes n}\big(\Gamma_\mathrm{orb}(\mathbf{X}_1,\dots,\mathbf{X}_n:(\mathbf{A}_i)_{i=1}^n\,;\,N,m,\delta)\big)
$$   
is Borel, and 
\begin{align*}
&(\gamma_{\mathrm{U}(N)}^{\otimes n}\otimes\mu)(\Phi_N^{-1}(\Gamma_R(\mathbf{X}_1\sqcup\cdots\sqcup\mathbf{X}_n\,;\,N,m,\delta))) \\ 
&= 
\int_{\prod_{i=1}^n ((M_N^\mathrm{sa})_R)^{r(i)}} \gamma_{\mathrm{U}(N)}^{\otimes n}\big(\Gamma_\mathrm{orb}(\mathbf{X}_1,\dots,\mathbf{X}_n:(\mathbf{A}_i)_{i=1}^n\,;\,N,m,\delta)\big)\,d\mu \\
&= 
\int_{\prod_{i=1}^n \Gamma_R(\mathbf{X}_i\,;\,N,m,\delta)} \gamma_{\mathrm{U}(N)}^{\otimes n}\big(\Gamma_\mathrm{orb}(\mathbf{X}_1,\dots,\mathbf{X}_n:(\mathbf{A}_i)_{i=1}^n\,;\,N,m,\delta)\big)\,d\mu  
\end{align*} 
holds for every probability measure $\mu \in \mathcal{P}(\prod_{i=1}^n(M_N^\mathrm{sa})^{r(i)})$. 
\end{lemma} 
\begin{proof} For each $(\mathbf{A}_i)_{i=1}^n \in \prod_{i=1}^n ((M_N^\mathrm{sa})_R)^{r(i)}$ it is trivial that the $(\mathbf{A}_i)_{i=1}^n$-section 
\begin{align*} 
\big\{ (U_i)_{i=1} \in \mathrm{U}(N)^n&\,|\, ((U_i)_{i=1}^n,(\mathbf{A}_i)_{i=1}^n) \in  \Phi_N^{-1}(\Gamma_R(\mathbf{X}_1\sqcup\cdots\sqcup\mathbf{X}_n\,;\,N,m,\delta))\big\} \\
&= \big\{((U_i)_{i=1}^n \in \mathrm{U}(N)^n\,|\,(U_i\mathbf{A}_i U_i^*)_{i=1}^n \in \Gamma_R(\mathbf{X}_1\sqcup\cdots\sqcup\mathbf{X}_n\,;\,N,m,\delta)\big\}
\end{align*} 
coincides with $\Gamma_\mathrm{orb}(\mathbf{X}_1,\dots,\mathbf{X}_n:(\mathbf{A}_i)_{i=1}^n\,;\,N,m,\delta)$, which clearly becomes the empty set if $(\mathbf{A}_i)_{i=1}^n \in \prod_{i=1}^n ((M_N^\mathrm{sa})_R)^{r(i)}\setminus\prod_{i=1}^n \Gamma_R(\mathbf{X}_i\,;\,N,m,\delta)$ as remarked before. Hence the desired assertion immediately follows by the Fubini theorem or the definition of product measures, see e.g.~\cite[Chap.~8]{Rudin:RedBook}.
\end{proof}

When $W^*(\mathbf{X}_i)$ is hyperfinite, one can choose multi-matrices $\Xi_i(N) = (\xi_{ij}(N))_{j=1}^{r(i)} \in (M_N^\mathrm{sa})^{r(i)}$ in such a way that $\Vert\xi_{ij}(N)\Vert_\infty \leq \Vert X_{ij}\Vert_\infty$ and $\lim_{N\to\infty}\frac{1}{N}\mathrm{Tr}_N(P(\Xi_i(N)))   = \tau(P(\mathbf{X}_i))$ for every $P \in \mathbb{C}\langle\mathbf{X}_i\rangle = \mathbb{C}\langle X_{i1},\dots,X_{ir(i)}\rangle$, the non-commutative polynomials in $X_{i1},\dots,X_{ir(i)}$, where $P(\Xi_i(N)) := P(\xi_{i1}(N),\dots,\xi_{ir(i)}(N))$ and $P(\mathbf{X}_i) := P(X_{i1},\dots,X_{ir(i)})$ as before.  

\begin{proposition}\label{P-2-3} {\rm(Agreement with the previous definition)} With the hyperfiniteness of each $W^*(\mathbf{X}_i)$, $1 \leq i \leq n$, and the notations above we have  
\begin{align}\label{Eq1}
&\chi_{\mathrm{orb},R}(\mathbf{X}_1,\dots,\mathbf{X}_n)  \notag\\
&= 
\lim_{\substack{m\to\infty \\ \delta\searrow0}} \limsup_{N\to\infty} 
\frac{1}{N^2} \log\Big(\gamma_{\mathrm{U}(N)}^{\otimes n}\big(\Gamma_\mathrm{orb}(\mathbf{X}_1,\dots,\mathbf{X}_n:(\Xi_i(N))_{i=1}^n\,;\,N,m,\delta)\big)\Big) 
\end{align} 
as long as $R \geq \max\{\Vert X_{ij}\Vert_\infty\,|\,1 \leq j \leq r(i),\, 1 \leq i \leq n\}$. The right-hand side is nothing less than the previous definition of $\chi_\mathrm{orb}(\mathbf{X}_1,\dots,\mathbf{X}_n)$, and hence the new definition agrees with the previous one when every $W^*(\mathbf{X}_i)$ is hyperfinite. 
\end{proposition}
\begin{proof} Let $m \in \mathbb{N}$ and $\delta>0$ be arbitrarily fixed. For all sufficiently large $N \in \mathbb{N}$ one has $\Xi_i(N) \in \Gamma_R(\mathbf{X}_i\,;\,N,m,\delta)$. With $\mu = \delta_{(\Xi_i(N))_{i=1}^n}$, a unit point mass, one has, by Lemma \ref{L-2-2},  
\begin{align*} 
\chi_{\mathrm{orb},R}(\mathbf{X}_1,\dots,\mathbf{X}_n\,;\,N,m,\delta) 
&\geq 
\chi_{\mathrm{orb},R}(\mathbf{X}_1,\dots,\mathbf{X}_n\,;\,N,m,\delta\,;\,\mu) \\
&= 
\log\Big(\gamma_{\mathrm{U}(N)}^{\otimes n}\big(\Gamma_\mathrm{orb}(\mathbf{X}_1,\dots,\mathbf{X}_n:(\Xi_i(N))_{i=1}^n\,;\,N,m,\delta)\big)\Big)
\end{align*}
for all sufficiently large $N \in \mathbb{N}$. It follows that $\chi_{\mathrm{orb},R}(\mathbf{X}_1,\dots,\mathbf{X}_n)$ is not smaller than the right-hand side of \eqref{Eq1}. Hence we need to prove that the reverse inequality; namely it suffices to prove that for each $m \in \mathbb{N}$ and $\delta>0$ there are $m' \in \mathbb{N}$, $\delta' > 0$ and $N_0 \in \mathbb{N}$ such that  
\begin{align}\label{Eq2} 
&\gamma_{\mathrm{U}(N)}^{\otimes n}\big(\Gamma_\mathrm{orb}(\mathbf{X}_1,\dots,\mathbf{X}_n:(\mathbf{A}_i)_{i=1}^n\,;\,N,m',\delta')\big) \notag\\
&\quad\quad 
\leq \gamma_{\mathrm{U}(N)}^{\otimes n}\big(\Gamma_\mathrm{orb}(\mathbf{X}_1,\dots,\mathbf{X}_n:(\Xi_i(N))_{i=1}^p\,;\,N,m,\delta)\big)
\end{align}
for every $N \geq N_0$ and $(\mathbf{A}_i)_{i=1}^n \in \prod_{i=1}^n \Gamma_R(\mathbf{X}_i\,;\,N,m',\delta')$. In fact, this and Lemma \ref{L-2-2} imply that 
\begin{align*} 
\chi_{\mathrm{orb},R}(\mathbf{X}_1,\dots,\mathbf{X}_n)
&\leq
\chi_{\mathrm{orb},R}(\mathbf{X}_1,\dots,\mathbf{X}_n\,;\,m',\delta') \\
&\leq \limsup_{N\to\infty} 
\frac{1}{N^2} \log\Big(\gamma_{\mathrm{U}(N)}^{\otimes n}\big(\Gamma_\mathrm{orb}(\mathbf{X}_1,\dots,\mathbf{X}_n:(\Xi_i(N))_{i=1}^n\,;\,N,m,\delta)\big)\Big), 
\end{align*}
which implies the desired inequality as taking the limit of the rightmost as $m\to\infty$ and $\delta \searrow 0$. 

Applying Jung's theorem \cite{Jung:MathAnn07} (see \cite[Lemma 2.1]{HiaiMiyamotoUeda:IJM09}) to all $\mathbf{X}_i$'s one can find $m' \in \mathbb{N}$ and $\delta'' > 0$ in such a way that for each $N \in \mathbb{N}$, each $1 \leq i \leq n$ and each $\mathbf{B}_1 = (B_{11},\dots,B_{1r(i)}), \mathbf{B}_2 = (B_{21},\dots,B_{2r(i)}) \in \Gamma_\infty(\mathbf{X}_i\,;\,N,m',\delta'')$ there is a single unitary $U = U_i(\mathbf{B}_1,\mathbf{B}_2) \in \mathrm{U}(N)$ such that $\Vert B_{1j} - UB_{2j} U^*\Vert_{1,\mathrm{tr}_N} < \delta/(2m(R+1)^{m-1})$ for every $1 \leq j \leq r(i)$. Let $N_0 \in \mathbb{N}$ be chosen so that every $\Xi_i(N)$ falls in $\Gamma_R(\mathbf{X}_i\,;\,N,m',\delta'')$ as long as $N \geq N_0$. Let $\delta'>0$ be such that $\delta' < \min\{\delta'', \delta/2\}$. Then for every $N \geq N_0$ and every $(\mathbf{A}_i)_{i=1}^n \in \prod_{i=1}^n \Gamma_R(\mathbf{X}_i\,;\,N,m',\delta')$ one can prove (see the proof of \cite[Lemma 4.2]{HiaiMiyamotoUeda:IJM09}) that 
\begin{align*}
&\Gamma_\mathrm{orb}(\mathbf{X}_1,\dots,\mathbf{X}_n:(\mathbf{A}_i)_{i=1}^n\,;\,N,m',\delta') \\
&\quad \subseteq 
\Gamma_\mathrm{orb}(\mathbf{X}_1,\dots,\mathbf{X}_n:(\Xi_i(N))_{i=1}^n\,;\,N,m,\delta)\cdot(V_1,\dots,V_n)
\end{align*}
holds with $V_i := U_i(\mathbf{A}_i,\Xi_i(N)) \in \mathrm{U}(N)$, $1 \leq i \leq n$. Hence \eqref{Eq2} follows due to the right-invariance of the Haar probability measure $\gamma_{\mathrm{U}(N)}$.  
\end{proof} 

The next proposition re-defines $\chi_{\mathrm{orb},R}(\mathbf{X}_1,\dots,\mathbf{X}_n)$ and shows a relationship between our new $\chi_\mathrm{orb}$ and Biane--Dabrowski's $\tilde{\chi}_\mathrm{orb}$ as promised before. 
    
\begin{proposition}\label{P-2-4} {\rm(Yet another definition of $\chi_\mathrm{orb}$)} With 
\begin{align}\label{Eq3}
&\bar{\chi}_{\mathrm{orb},R}(\mathbf{X}_1,\dots,\mathbf{X}_n\,;\,N,m,\delta) \notag\\
&\quad\quad:= 
\sup_{\mathbf{A}_i \in (M_N^\mathrm{sa})_R^{r(i)}}\log\Big(\gamma_{\mathrm{U}(N)}^{\otimes n}\big(\Gamma_\mathrm{orb}(\mathbf{X}_1,\dots,\mathbf{X}_n:(\mathbf{A}_i)_{i=1}^n\,;\,N,m,\delta)\big)\Big) \\
&\quad\quad\,= 
\sup_{\mathbf{A}_i \in \Gamma_R(\mathbf{X}_i\,;\,N,m,\delta)}\log\Big(\gamma_{\mathrm{U}(N)}^{\otimes n}\big(\Gamma_\mathrm{orb}(\mathbf{X}_1,\dots,\mathbf{X}_n:(\mathbf{A}_i)_{i=1}^n\,;\,N,m,\delta)\big)\Big) \notag
\end{align}
{\rm(}defined to be $-\infty$ if some $\Gamma_R(\mathbf{X}_i\,;\,N,m,\delta) = \emptyset${\rm)} we have  
\begin{equation}\label{Eq4}
\chi_{\mathrm{orb},R}(\mathbf{X}_1,\dots,\mathbf{X}_n) 
= 
\lim_{\substack{m\to\infty \\ \delta\searrow0}} \limsup_{N\to\infty} 
\frac{1}{N^2}\bar{\chi}_{\mathrm{orb},R}(\mathbf{X}_1,\dots,\mathbf{X}_n\,;\,N,m,\delta). 
\end{equation} 
In particular, our $\chi_\mathrm{orb}$ is not greater than Biane--Dabrowski's $\tilde{\chi}_\mathrm{orb}$ in general, and both coincide at least when $W^*(\mathbf{X}_1\sqcup\cdots\sqcup\mathbf{X}_n)$ is a factor. 
\end{proposition}
\begin{proof} By definition the second equality in \eqref{Eq3} is obvious. 

By Lemma \ref{L-2-2} we have 
$$
\chi_{\mathrm{orb},R}(\mathbf{X}_1,\dots,\mathbf{X}_n\,;\,N,m,\delta\,;\,\mu) 
\leq
\bar{\chi}_{\mathrm{orb},R}(\mathbf{X}_1,\dots,\mathbf{X}_n\,;\,N,m,\delta)
$$
since $\mu$ is a probability measure and $t \in [0,+\infty) \mapsto \log t \in [-\infty,+\infty)$ with $\log 0 := -\infty$ is monotone. This immediately implies inequality `$\leq$' in \eqref{Eq4}. Hence it suffices to prove inequality `$\geq$' in \eqref{Eq4} when the right-hand side is finite, i.e., $\neq -\infty$. Since the limit as $m\to\infty$ and $\delta\searrow0$ is actually the infimum over $m \in \mathbb{N}$ and $\delta>0$, we may and do assume that for every $m \in \mathbb{N}$ and $\delta>0$ there is a subsequence $N_1 < N_2 < \cdots$ so that $\bar{\chi}_{\mathrm{orb},R}(\mathbf{X}_1,\dots,\mathbf{X}_n\,;\,N_k,m,\delta) > -\infty$ for all $k \in \mathbb{N}$ and  
$$
\limsup_{N\to\infty}\frac{1}{N^2}\bar{\chi}_{\mathrm{orb},R}(\mathbf{X}_1,\dots,\mathbf{X}_n\,;\,N,m,\delta) = \lim_{k\to\infty}\frac{1}{N_k^2}\bar{\chi}_{\mathrm{orb},R}(\mathbf{X}_1,\dots,\mathbf{X}_n\,;\,N_k,m,\delta).
$$ 
For each $k \in \mathbb{N}$ one can choose $\mathbf{A}_i^{(k)} \in ((M_{N_k}^\mathrm{sa})_R)^{r(i)}$, $1 \leq i \leq n$, in such a way that $\mu_{N_k} := \delta_{(\mathbf{A}_i^{(k)})_{i=1}^n} \in \mathcal{P}(\prod_{i=1}^n((M_{N_k}^\mathrm{sa})_R)^{r(i)})$, a unit point mass, satisfies that 
\begin{align*} 
\bar{\chi}_{\mathrm{orb},R}(\mathbf{X}_1,\dots,\mathbf{X}_n\,;\,N_k,m,\delta) - 1
&\leq 
\chi_{\mathrm{orb},R}(\mathbf{X}_1,\dots,\mathbf{X}_n\,;\,N_k,m,\delta\,;\,\mu_{N_k}) \\
&\leq 
\chi_{\mathrm{orb},R}(\mathbf{X}_1,\dots,\mathbf{X}_n\,;\,N_k,m,\delta), 
\end{align*}
and hence
\begin{align*}
\limsup_{N\to\infty} 
\frac{1}{N^2}\bar{\chi}_{\mathrm{orb},R}(\mathbf{X}_1,\dots,\mathbf{X}_n\,;\,N,m,\delta)
&= 
\lim_{k\to\infty}\frac{1}{N_k^2} \bar{\chi}_{\mathrm{orb},R}(\mathbf{X}_1,\dots,\mathbf{X}_n\,;\,N_k,m,\delta) \\
&\leq
\limsup_{k\to\infty} 
\frac{1}{N_k^2} 
\Big(\chi_{\mathrm{orb},R}(\mathbf{X}_1,\dots,\mathbf{X}_n\,;\,N_k ,m,\delta) + 1\Big) \\
&\leq 
\chi_{\mathrm{orb},R}(\mathbf{X}_1,\dots,\mathbf{X}_n\,;\,m,\delta). 
\end{align*} 
Therefore, the desired inequality follows. 

The rest is an immediate consequence from the proof of \cite[Theorem 7.3 (7)]{BianeDabrowski:AdvMath13}.  
\end{proof} 

The next lemma shows that the procedure of cut-off $R>0$ that appears in the definition of $\chi_\mathrm{orb}$ is inessential. The proof below technically originates in \cite[Proposition 2.4]{Voiculescu:InventMath94} and \cite[Lemma 2.3]{BelinschiBercovici:PacificJMath03}. 

\begin{lemma}\label{L-2-5} We have $\chi_\mathrm{orb}(\mathbf{X}_1,\dots,\mathbf{X}_n) = \chi_{\mathrm{orb},R}(\mathbf{X}_1,\dots,\mathbf{X}_n)$ if $R \gneqq \rho := \max\{\Vert X_{ij}\Vert_\infty\,|\,1 \leq j \leq r(i),\, 1 \leq i \leq n\}$ possibly with $R=\infty$.  
\end{lemma}
\begin{proof} Define the function $f : \mathbb{R} \to [-1,1]$ by $f(t) = t$ for $t \in [-1,1]$, $f(t) = -1$ for $t < -1$, and $f(t) = 1$ for $t > 1$. Choose an arbitrary $R \gneqq \rho$, and set $f_R(t) := Rf(t/R)$ for $t \in \mathbb{R}$. 

Let $m \in \mathbb{N}$ and $\delta>0$ be arbitrarily fixed. Set $L := \max\{(\rho^{2m}+1)^{1/2m},R\} >1$. Choose $0<\delta'<\min\{1,\delta/2\}$. One can choose an even $m' \in \mathbb{N}$ in such a way that $m' \geq m$ and $R\big((\rho/R)^{m'} + \delta'/R^{m'}\big)^{1/m} < \delta/(2mL^{m-1})$. Let $\mathbf{A}_i = (A_{ij})_{j=1}^{r(i)} \in \Gamma_\infty(\mathbf{X}_i\,;\,N,m',\delta')$ be given. Then one has $\mathrm{tr}_N(A_{ij}^{m''}) \leq \rho^{m''} + \delta'$ if $m''$ is even and not greater than $m'$. Since $\mathrm{tr}_N(I) = 1$, one has $\Vert-\Vert_{p,\mathrm{tr}_N} \leq \Vert-\Vert_{q,\mathrm{tr}_N}$ if $1 \leq p \leq q$, and hence $\Vert A_{ij}\Vert_{p,\mathrm{tr}_N} \leq \Vert A_{ij}\Vert_{2m,\mathrm{tr}_N} \leq L$ as long as $p \leq 2m$. If $\lambda_1,\dots,\lambda_N$ are the eigenvalues of $A_{ij}$ with counting multiplicities, one has 
\begin{align*}
\Vert A_{ij} - f_R(A_{ij})\Vert_{m,\mathrm{tr}_N} 
&\leq R\Big(\frac{1}{N}\sum_{|\lambda_j| > R} \Big|\frac{\lambda_j}{R}\Big|^m\Big)^{1/m} 
\leq R\Big(\frac{1}{N}\sum_{|\lambda_j| > R} \Big(\frac{\lambda_j}{R}\Big)^{m'}\Big)^{1/m} \\
&\leq R\big(\mathrm{tr}_N(A_{ij}^{m'})/R^{m'}\big)^{1/m} 
\leq R\big((\rho/R)^{m'} + \delta'/R^{m'}\big)^{1/m} < \frac{\delta}{2mL^{m-1}}.
\end{align*} 
ThereforeCif $(U_i)_{i=1}^n \in \Gamma_\mathrm{orb}(\mathbf{X}_1,\dots,\mathbf{X}_n\,:(\mathbf{A}_i)_{i=1}^n\,;\,N,m',\delta')$, then one has
\begin{align*} 
&|\mathrm{tr}_N(U_{i_1}f_R(A_{i_1 j_1})U_{i_1}^*\cdots U_{i_l}f_R(A_{i_l j_l})U_{i_l}^*)-\tau(X_{i_l j_1}\cdots X_{i_l j_l})| \\
&\leq |\mathrm{tr}_N(U_{i_1}f_R(A_{i_1 j_1})U_{i_1}^*\cdots U_{i_l}f_R(A_{i_l j_l})U_{i_l}^*) - \mathrm{tr}_N(U_{i_1}A_{i_1 j_1}U_{i_1}^*\cdots U_{i_l}^*A_{i_1 j_l}U_{i_l}^*)| \\
&\quad\quad\quad\quad+ |\mathrm{tr}_N(U_{i_1} A_{i_1 j_1}U_{i_1}^* \cdots U_{i_l} A_{i_l j_l} U_{i_l}^*) - \tau(X_{i_1 j_1}\cdots X_{i_l j_l})| \\
&< \sum_{k=1}^l L^{l-1}\Vert U_{i_k} (f_R(A_{i_k j_k})-A_{i_k j_k})U_{i_k}^*\Vert_{l,\mathrm{tr}_N} + \delta' 
\leq mL^{m-1}\frac{\delta}{2mL^{m-1}} + \frac{\delta}{2} = \delta 
\end{align*}
for every $1 \leq j_k \leq r(i_k)$, $1 \leq i_k \leq n$, $1 \leq k \leq l$ and $l \leq m$, where the second inequality is shown by the so-called generalized H\"{o}lder inequality and the third follows from the previous estimate together with the unitary invariance of $p$-norms. Hence $(U_i)_{i=1}^n \in \Gamma_\mathrm{orb}(\mathbf{X}_1,\dots,\mathbf{X}_n\,:\,(f_R(\mathbf{A}_i))_{i=1}^n\,;\,N,m,\delta)$. Consequently, we get
$$
\bar{\chi}_{\mathrm{orb},\infty}(\mathbf{X}_1,\dots,\mathbf{X}_n\,;\,N,m',\delta') \leq 
\bar{\chi}_{\mathrm{orb},R}(\mathbf{X}_1,\dots,\mathbf{X}_n\,;\,N,m,\delta).  
$$
This implies that $\chi_{\mathrm{orb},\infty}(\mathbf{X}_1,\dots,\mathbf{X}_n) \leq \chi_{\mathrm{orb},R}(\mathbf{X}_1,\dots,\mathbf{X}_n)$ by Proposition \ref{P-2-4}. Hence we are done. 
\end{proof}    

Here are almost all the expected properties of $\chi_\mathrm{orb}$. 

\begin{theorem}\label{T-2-6} We have\,{\rm:}
\begin{itemize}
\item[(1)] $\chi_\mathrm{orb}(\mathbf{X}_1,\dots,\mathbf{X}_n) \leq 0$. 
\item[(2)] $\chi_\mathrm{orb}(\mathbf{X}_1,\dots,\mathbf{X}_n) = -\infty$ if $\mathbf{X}_1\sqcup\cdots\sqcup\mathbf{X}_n$ does not have finite dimensional approximates {\rm(}have f.d.a.~in short{\rm)} in the sense of \cite[Definition 3.1]{Voiculescu:IMRN98}. 
\item[(3)] $\chi_\mathrm{orb}(\mathbf{X}) = 0$ if $\mathbf{X}$ has f.d.a.{\rm;} otherwise $-\infty$. 
\item[(4)] $\chi_\mathrm{orb}(\mathbf{X}_1,\dots,\mathbf{X}_{n'},\mathbf{X}_{n'+1},\dots,\mathbf{X}_n) \leq \chi_\mathrm{orb}(\mathbf{X}_1,\dots,\mathbf{X}_{n'}) + \chi_\mathrm{orb}(\mathbf{X}_{n'+1},\dots,\mathbf{X}_n)$. 
\item[(5)] If $\mathbf{X}_1^{(k)}\sqcup\cdots\sqcup\mathbf{X}_n^{(k)} \longrightarrow \mathbf{X}_1\sqcup\cdots\sqcup\mathbf{X}_n$ in distribution as $k\to\infty$, then 
$$
\chi_{\mathrm{orb},R}(\mathbf{X}_1,\dots,\mathbf{X}_n) \geq \limsup_{k\to\infty}
\chi_{\mathrm{orb},R}(\mathbf{X}_1^{(k)},\dots,\mathbf{X}_n^{(k)})
$$
for every $R > 0$ possibly with $R=\infty$. Hence the same assertion holds true for $\chi_\mathrm{orb}$. 
\item[(6)] If $\mathbf{Y}_i = (Y_{ij})_{j=1}^{r'(i)} \subseteq W^*(\mathbf{X}_i)$, $1 \leq i \leq n$, then $\chi_\mathrm{orb}(\mathbf{X}_1,\dots,\mathbf{X}_n) \leq \chi_\mathrm{orb}(\mathbf{Y}_1,\dots,\mathbf{Y}_n)$. In particular, $\chi_\mathrm{orb}(\mathbf{X}_1,\dots,\mathbf{X}_n)$ depends only on the von Neumann subalgebras $W^*(\mathbf{X}_i)$ generated by $\mathbf{X}_i$ in $\mathcal{M}$.
\item[(7)] If $\mathbf{X}_1$ and $\{\mathbf{X}_2,\dots,\mathbf{X}_n\}$ are freely independent in $(\mathcal{M},\tau)$, then $\chi_\mathrm{orb}(\mathbf{X}_1,\mathbf{X}_2,\dots,\mathbf{X}_n) = \chi_\mathrm{orb}(\mathbf{X}_1) + \chi_\mathrm{orb}(\mathbf{X}_2,\dots,\mathbf{X}_n)$, which  turns out to be $\chi_\mathrm{orb}(\mathbf{X}_2,\dots,\mathbf{X}_n)$ when $\mathbf{X}_1$ has f.d.a. 
\item[(8)] $\chi_\mathrm{orb}(\mathbf{X}_1,\dots,\mathbf{X}_n) = 0$ if and only if the $\mathbf{X}_i$'s have f.d.a.~and are freely independent in $(\mathcal{M},\tau)$. 
\end{itemize}   
\end{theorem} 
\begin{proof} (1)--(4) are trivial or straightforward, and left to the reader. 

(5) As in the proof of \cite[Proposition 2.6]{Voiculescu:InventMath94}, for given $m \in \mathbb{N}$ and $\delta>0$ we have 
$$
\Gamma_R(\mathbf{X}_1^{(k)}\sqcup\cdots\sqcup\mathbf{X}_n^{(k)};N,m,\delta) 
\subseteq 
\Gamma_R(\mathbf{X}_1\sqcup\cdots\sqcup\mathbf{X}_n;N,m,2\delta)
$$
for all sufficiently large $k \in \mathbb{N}$ (this is valid even for $R=\infty$), and hence   
\begin{align*} 
\chi_{\mathrm{orb},R}(\mathbf{X}_1^{(k)},\dots,\mathbf{X}_n^{(k)}\,;\,N,m,\delta;\mu) 
&\leq 
\chi_{\mathrm{orb},R}(\mathbf{X}_1,\dots,\mathbf{X}_n\,;\,N,m,2\delta\,;\,\mu) \\
&\leq 
\chi_{\mathrm{orb},R}(\mathbf{X}_1,\dots,\mathbf{X}_n\,;\,N,m,2\delta)
\end{align*}
for all the same large $k \in \mathbb{N}$. Thus 
$$
\limsup_{k\to\infty}\chi_{\mathrm{orb},R}(\mathbf{X}_1^{(k)},\dots,\mathbf{X}_n^{(k)}) 
\leq 
\chi_{\mathrm{orb},R}(\mathbf{X}_1,\dots,\mathbf{X}_n\,;\,m,2\delta),
$$
implying the desired assertion. 

(6) Choose arbitrary $m \in \mathbb{N}$ and $\delta>0$. By the Kaplansky density theorem one can choose tuples of self-adjoint polynomials $\mathbf{P}_i = (P_{ij})_{j=1}^{r'(i)} \in \prod_{j=1}^{r'(i)}\mathbb{C}\langle\mathbf{X}_i\rangle$ in such a way that $\Vert P_{ij}(\mathbf{X}_i)\Vert_\infty \leq \Vert Y_{ij}\Vert_\infty$ and $\Vert Y_{ij} - P_{ij}(\mathbf{X}_i)\Vert_{1,\tau} \leq \Vert Y_{ij} - P_{ij}(\mathbf{X}_i)\Vert_{2,\tau} < \delta/(2mL^{m-1})$ with $L:=\max\{\Vert Y_{ij}\Vert_\infty\,|\,1 \leq j \leq r(i), 1 \leq i \leq n\} + 1$. Looking at the $P_{ij}$'s one can also choose $m' \in \mathbb{N}$ and $\delta' > 0$ in such a way that if $(\mathbf{A}_i)_{i=1}^n \in \Gamma_\infty(\mathbf{X}_1\sqcup\cdots\sqcup\mathbf{X}_n\,;\,N,m',\delta')$, then 
$$
|\mathrm{tr}_N(P_{i_1 j_1}(\mathbf{A}_{i_1})\cdots P_{i_l j_l}(\mathbf{A}_{i_l})) - \tau(P_{i_1 j_1}(\mathbf{X}_{i_1})\cdots P_{i_l j_l}(\mathbf{X}_{i_l}))| < \frac{\delta}{2}
$$
for every $1 \leq j_k \leq r'(i)$, $1 \leq i_k \leq n$, $1 \leq k \leq l$ and $1 \leq l \leq m$. Remark here that $m', \delta'$ are independent of $N$ by the way of finding them. If $(\mathbf{A}_i)_{i=1}^n \in \Gamma_\infty(\mathbf{X}_1\sqcup\cdots\sqcup\mathbf{X}_n\,;\,N,m',\delta')$, then 
\begin{align*}
|\mathrm{tr}_N&(P_{i_1 j_1}(\mathbf{A}_{i_1})\cdots P_{i_l j_l}(\mathbf{A}_{i_l})) - \tau(Y_{i_1 j_1}\cdots Y_{i_l j_l})| \\
&\leq
\frac{\delta}{2} 
+ |\tau(P_{i_1 j_1}(\mathbf{X}_{i_1})\cdots P_{i_l j_l}(\mathbf{X}_{i_l})) - \tau(Y_{i_1 j_1}\cdots Y_{i_l j_l})| \\
&\leq
\frac{\delta}{2} + l L^{l-1} \max\{\Vert P_{ij}(\mathbf{X}_i) - Y_{ij}\Vert_{\tau,1}\,|\,1 \leq j \leq r'(i), 1 \leq i \leq n\} < \frac{\delta}{2} + \frac{\delta}{2} = \delta,
\end{align*}
implying $(\mathbf{P}_i(\mathbf{A}_i))_{i=1}^n = ((P_{ij}(\mathbf{A}_i))_{j=1}^{r'(i)})_{i=1}^n \in \Gamma_\infty(\mathbf{Y}_1\sqcup\cdots\sqcup\mathbf{Y}_n\,;\,N,m,\delta)$. Hence, if $(U_i)_{i=1}^n \in \Gamma_\mathrm{orb}(\mathbf{X}_1,\dots,\mathbf{X}_n:(\mathbf{A}_i)_{i=1}^n\,;\,N,m',\delta')$, i.e., $(U_i\mathbf{A}_i U_i^*)_{i=1}^n \in \Gamma_\infty(\mathbf{X}_1\sqcup\cdots\sqcup\mathbf{X}_n\,;\,N,m', \delta')$, then $(U_i\mathbf{P}_i(\mathbf{A}_i)U_i^*)_{i=1}^n = (\mathbf{P}_i(U_i\mathbf{A}_i U_i^*))_{i=1}^n \in \Gamma_\infty(\mathbf{Y}_1\sqcup\cdots\sqcup\mathbf{Y}_n\,;\,N,m,\delta)$ by the above consideration so that $(U_i)_{i=1}^n \in \Gamma_\mathrm{orb}(\mathbf{Y}_1,\dots,\mathbf{Y}_n:(\mathbf{P}_i(\mathbf{A}_i))_{i=1}^n\,;\,N,m,\delta)$. Consequently 
$$ 
\bar{\chi}_{\mathrm{orb},\infty}(\mathbf{X}_1,\dots,\mathbf{X}_n\,;\,N,m',\delta') \leq 
\bar{\chi}_{\mathrm{orb},\infty}(\mathbf{Y}_1,\dots,\mathbf{Y}_n\,;\,N,m,\delta).
$$ 
Hence the desired inequality immediately follows thanks to Proposition \ref{P-2-4} and Lemma \ref{L-2-5}. The latter assertion is immediate. 

(7) We may and do assume that $\mathbf{X}_1$ has f.d.a.~or equivalently $\chi_\mathrm{orb}(\mathbf{X}_1) = 0$ by the above (3) and moreover that $\chi_\mathrm{orb}(\mathbf{X}_2,\dots,\mathbf{X}_n) > -\infty$; otherwise $\chi_\mathrm{orb}(\mathbf{X}_1,\dots,\mathbf{X}_n) = -\infty = \chi_\mathrm{orb}(\mathbf{X}_1) + \chi_\mathrm{orb}(\mathbf{X}_2,\dots,\mathbf{X}_n)$ by the above (4). Here we need to prove only that $\chi_\mathrm{orb}(\mathbf{X}_1,\dots,\mathbf{X}_n) \geq \chi_\mathrm{orb}(\mathbf{X}_2,\dots,\mathbf{X}_n)$ thanks to the above (4) again. 

Let $R > \max\{\Vert X_{ij}\Vert_\infty\,|\,1\leq j\leq r(i), 1 \leq i \leq n\}$ be fixed, and $m \in \mathbb{N}$, $\delta>0$ be given. One can choose $\delta'>0$ in such a way that if $\mathbf{A}_1 \in \Gamma_R(\mathbf{X}_1\,;\,N,m,\delta')$ and $(\mathbf{A}_i)_{i=2}^n \in \Gamma_R(\mathbf{X}_2\sqcup\cdots\sqcup\mathbf{X}_n\,;\,N,m,\delta')$ are $(m,\delta')$-free in the sense of \cite[Definition 2.10]{Voiculescu:IMRN98}, then $(\mathbf{A}_i)_{i=1}^n \in \Gamma_R(\mathbf{X}_1\sqcup\cdots\sqcup\mathbf{X}_n\,;\,N,m,\delta)$. Since $\chi_{\mathrm{orb},R}(\mathbf{X}_2,\dots,\mathbf{X}_n) = \chi_\mathrm{orb}(\mathbf{X}_2,\dots,\mathbf{X}_n) > -\infty$ (due to Lemma \ref{L-2-5}), one can choose a subsequence $N_1 < N_2 < \cdots$ in such a way that $\bar{\chi}_{\mathrm{orb},R}(\mathbf{X}_2,\dots,\mathbf{X}_n\,;\,N_k,m,\delta') > -\infty$ for all $k \in \mathbb{N}$ and 
$$
\limsup_{N\to\infty}\frac{1}{N^2}\bar{\chi}_{\mathrm{orb},R}(\mathbf{X}_2,\dots,\mathbf{X}_n\,;\,N,m,\delta') = 
\lim_{k\to\infty}\frac{1}{N_k^2}\bar{\chi}_{\mathrm{orb},R}(\mathbf{X}_2,\dots,\mathbf{X}_n\,;\,N_k,m,\delta').
$$ 
For each $k \in \mathbb{N}$ we can choose $\mathbf{A}_i^{(k)} \in ((M_N^\mathrm{sa})_R)^{r(i)}$, $2 \leq i \leq n$, in such a way that
$$
\bar{\chi}_{\mathrm{orb},R}(\mathbf{X}_2,\dots,\mathbf{X}_n\,;\,N_k,m,\delta') - 1
\leq 
\log\Big(\gamma_{\mathrm{U}(N)}^{\otimes n-1}\big(\Gamma_\mathrm{orb}(\mathbf{X}_2,\dots,\mathbf{X}_n:(\mathbf{A}_i^{(k)})_{i=2}^n\,;\,N_k,m,\delta')\big)\Big),   
$$
and in particular, $\gamma_{\mathrm{U}(N_k)}^{\otimes n-1}\big(\Gamma_\mathrm{orb}(\mathbf{X}_1,\dots,\mathbf{X}_n:(\mathbf{A}_i^{(k)})_{i=2}^n\,;\,N_k,m,\delta')\big) \gneqq 0$. Taking a subsequence of $N_k$ if necessary we can choose $\mathbf{A}_1^{(k)} \in \Gamma_R(\mathbf{X}_1\,;\,N_k,m,\delta')$ for every $k \in \mathbb{N}$, since $\mathbf{X}_1$ has f.d.a. For simplicity, write 
\begin{align*} 
\Psi(N_k,m,\delta) &:= \Gamma_\mathrm{orb}(\mathbf{X}_1,\dots,\mathbf{X}_n:(\mathbf{A}_i^{(k)})_{i=1}^n\,;\,N_k,m,\delta), \\ 
\Phi(N_k,m,\delta') &:= \Gamma_\mathrm{orb}(\mathbf{X}_2,\dots,\mathbf{X}_n:(\mathbf{A}_i^{(k)})_{i=2}^n\,;\,N_k,m,\delta').
\end{align*}  
Define 
$$
\Omega(N_k,m,\delta') := 
\{ (U_i)_{i=1}^n \in \mathrm{U}(N)^n\,|\,\text{$U_1\mathbf{A}_1^{(k)}U_1^*$, $(U_i\mathbf{A}_i^{(k)} U_i^*)_{i=2}^n$ are $(m,\delta')$-free}\} 
$$
and 
$$
\mu_{N_k} := \frac{1}{\gamma_{\mathrm{U}(N_k)}^{\otimes n-1}(\Phi(N_k,m,\delta'))}\gamma_{\mathrm{U}(N_k)}^{\otimes n-1}\!\upharpoonright_{\Phi(N_k,m,\delta')}\, \in\, \mathcal{P}\big(\Phi(N_k,m,\delta')\big).
$$ 
Let $(U_i)_{i=1}^n \in \mathrm{U}(N_k)^n$ be given. If $(U_i)_{i=2}^n \in \Phi(N_k,m,\delta')$ and $(U_i)_{i=1}^n \in \Omega(N_k,m,\delta')$, then $(U_i\mathbf{A}_i^{(k)}U_i^*)_{i=1}^n \in \Gamma_R(\mathbf{X}_1,\dots,\mathbf{X}_n\,;\,N_k,m,\delta)$ as seen before, which is equivalent to that $(U_i)_{i=1}^n \in \Psi(N_k,m,\delta)$. Consequently, $\big(\mathrm{U}(N_k)\times\Phi(N_k,m,\delta')\big)\,\cap\,\Omega(N_k,m,\delta') \subseteq \Psi(N_k,m,\delta)$. By \cite[Corollary 2.13]{Voiculescu:IMRN98} there exists $N_0 \in \mathbb{N}$ such that 
$$
\gamma_{\mathrm{U}(N_k)}\big(\{U_1 \in \mathrm{U}(N_k)\,|\,\text{$(U_i)_{i=1}^n \in \Omega(N_k,m,\delta')$}\}\big) > \frac{1}{2}
$$
for every $N_k \geq N_0$ and $(U_i)_{i=2}^n \in \mathrm{U}(N_k)^{n-1}$. Hence we have
\begin{align*} 
&\frac{\gamma_{\mathrm{U}(N_k)}^{\otimes n}\big(\Psi(N_k,m,\delta)\big)}{\gamma_{\mathrm{U}(N_k)}^{\otimes n-1}\big(\Phi(N_k,m,\delta')\big)} 
\geq 
(\gamma_{\mathrm{U}(N_k)}\otimes\mu_{N_k})\big(\Psi(N_k,m,\delta)\big) \\
&\geq 
(\gamma_{\mathrm{U}(N_k)}\otimes\mu_{N_k})\big(\big(\mathrm{U}(N_k)\times\Phi(N_k,m,\delta')\big)\cap\Omega(N_k,m,\delta')\big) \\
&= 
\int_{\Phi(N_k,m,\delta')} \gamma_{\mathrm{U}(N_k)}\big(\{U_1 \in \mathrm{U}(N_k)\,|\,\text{$(U_i)_{i=1}^n \in \Omega(N_k,m,\delta)$}\}\big)\,\mu_{N_k}(d(U_i)_{i=2}^n) > \frac{1}{2}
\end{align*}
whenever $N_k \geq N_0$ by the Fubini theorem as in Lemma \ref{L-2-2}. Therefore, we have 
\begin{align*}
\bar{\chi}_{\mathrm{orb},R}(\mathbf{X}_2,\dots,\mathbf{X}_n\,;\,N_k,m,\delta') 
&\leq  
\log\big(\gamma_{\mathrm{U}(N_k)}^{\otimes n-1}\big(\Phi(N_k,m,\delta')\big)\big) + 1 \\
&< 
\log\big(2\gamma_{\mathrm{U}(N_k)}^{\otimes n}\big(\Psi(N_k,m,\delta)\big)\big) + 1 \\
&\leq \bar{\chi}_{\mathrm{orb},R}(\mathbf{X}_1,\dots,\mathbf{X}_n\,;\,N_k,m,\delta) + \log2 + 1  
\end{align*} 
whenever $N_k \geq N_0$, and thus
\begin{align*} 
\chi_{\mathrm{orb},R}(\mathbf{X}_2,\dots,\mathbf{X}_n) &\leq 
\limsup_{N\to\infty}\frac{1}{N^2}\bar{\chi}_{\mathrm{orb},R}(\mathbf{X}_2,\dots,\mathbf{X}_n\,;\,N,m,\delta') \\
&= 
\lim_{k\to\infty}\frac{1}{N_k^2}\bar{\chi}_{\mathrm{orb},R}(\mathbf{X}_2,\dots,\mathbf{X}_n\,;\,N_k,m,\delta') \\ 
&\leq 
\limsup_{k\to\infty}\frac{1}{N_k^2}\Big(\bar{\chi}_{\mathrm{orb},R}(\mathbf{X}_1,\dots,\mathbf{X}_n\,;\,N_k,m,\delta) + \log2 + 1\Big) \\
&= 
\limsup_{k\to\infty}\frac{1}{N_k^2}\bar{\chi}_{\mathrm{orb},R}(\mathbf{X}_1,\dots,\mathbf{X}_n\,;\,N_k,m,\delta) \\
&\leq 
\limsup_{N\to\infty}\frac{1}{N^2}\bar{\chi}_{\mathrm{orb},R}(\mathbf{X}_1,\dots,\mathbf{X}_n\,;\,N,m,\delta). 
\end{align*} 
Hence the desired inequality follows thanks to Lemma \ref{L-2-5}. 

(8) The `if' part follows from the above (7). Since $-\tilde{\chi}_\mathrm{orb}(\mathbf{X}_1,\dots,\mathbf{X}_n) \leq -\chi_\mathrm{orb}(\mathbf{X}_1,\dots,\mathbf{X}_n)$ due to Proposition \ref{P-2-4}, the orbital version of Talagrand's inequality \cite[Theorem 7.3 (9)]{BianeDabrowski:AdvMath13} also holds for our $\chi_\mathrm{orb}$. Hence the `only if' is immediate. A direct proof of the inequality can be given in the almost same way as in \cite[Proposition 4.4 (8)]{HiaiMiyamotoUeda:IJM09} and is a bit simpler than that for $\tilde{\chi}_\mathrm{orb}$; it will be outlined in the Appendix for the reader's convenience. 
\end{proof} 

The next corollary strengthens Lemma \ref{L-2-5}. 

\begin{corollary}\label{C-2-7} We have $\chi_\mathrm{orb}(\mathbf{X}_1,\dots,\mathbf{X}_n) = \chi_{\mathrm{orb},R}(\mathbf{X}_1,\dots,\mathbf{X}_n)$ with $R := \max\{\Vert X_{ij}\Vert_\infty\,|\,1\leq j \leq r(i), 1 \leq i \leq n\}$. 
\end{corollary}
\begin{proof} With $t\mathbf{X}_i := (tX_{i1},\dots,tX_{ir(i)})$, $0 < t < 1$, we have 
\begin{align*} 
\chi_{\mathrm{orb},R}(\mathbf{X}_1,\dots,\mathbf{X}_n) 
&\geq 
\limsup_{t\nearrow1}\chi_{\mathrm{orb},R}(t\mathbf{X}_1,\dots,t\mathbf{X}_n) \quad \text{(by Theorem \ref{T-2-6} (5))} \\
&=
\limsup_{t\nearrow1}\chi_\mathrm{orb}(t\mathbf{X}_1,\dots,t\mathbf{X}_n) \quad \text{(by Lemma \ref{L-2-5})} \\
&=
\limsup_{t\nearrow1}\chi_\mathrm{orb}(\mathbf{X}_1,\dots,\mathbf{X}_n) \quad \text{(by Theorem \ref{T-2-6} (6))} \\
&= 
\chi_\mathrm{orb}(\mathbf{X}_1,\dots,\mathbf{X}_n).
\end{align*} 
Hence we are done. 
\end{proof}  

The next proposition relates $\chi_\mathrm{orb}$ with free entropy $\chi$. Unfortunately we do not know whether or not equality holds in the proposition below except the case that every $\mathbf{X}_i$ consists of a single random variable, see \cite[Theorem 2.6]{HiaiMiyamotoUeda:IJM09}. This issue will be discussed further in Remark \ref{R-2-9}.        

\begin{proposition}\label{P-2-8} We have $\chi(\mathbf{X}_1\sqcup\cdots\sqcup\mathbf{X}_n) \leq \chi_\mathrm{orb}(\mathbf{X}_1,\dots,\mathbf{X}_n) + \sum_{i=1}^n \chi(\mathbf{X}_i)$. 
\end{proposition}
\begin{proof} By \cite[Proposition 2.3]{Voiculescu:InventMath94} we may and do assume that all $\chi(\mathbf{X}_i)$'s are finite, which implies that all $\mathbf{X}_i$'s have f.d.a., see \cite[Definition 3.1 and Remark 3.2]{Voiculescu:InventMath94}. Let $R > \max\{\Vert X_{ij}\Vert_\infty\,|\,1 \leq j \leq r(i), 1 \leq i \leq n\}$ be fixed, and for every $m \in \mathbb{N}$ and $\delta > 0$ there is $N_0 \in \mathbb{N}$ so that $\Gamma_R(\mathbf{X}_i\,;\,N,m,\delta) \neq \emptyset$ for all $1 \leq i \leq n$ and all $N \geq N_0$. Since each $\Gamma_R(\mathbf{X}_i\,;\,N,m,\delta)$ is clearly open, we observe that the Lebesgue measure $\Lambda_N^{\otimes r(i)}(\Gamma_R(\mathbf{X}_i\,;\,N,m,\delta))$ is nonzero for every $1 \leq i \leq n$ and every $N \geq N_0$. Hence we have the probability measures  
\begin{align}\label{Eq5}
\nu_R(N,m,\delta)  
:&= \sideset{}{^\otimes}\prod_{1\leq i \leq n}\left(\frac{1}{\Lambda_N^{\otimes r(i)}(\Gamma_R(\mathbf{X}_i\,;\,N,m,\delta))}\Lambda_N^{\otimes r(i)}\!\upharpoonright_{\Gamma_R(\mathbf{X}_i\,;\,N,m,\delta)}\right) \notag\\
&= 
\frac{1}{\prod_{i=1}^n\Lambda_N^{\otimes r(i)}(\Gamma_R(\mathbf{X}_i\,;\,N,m,\delta))}\Lambda_N^{\otimes(r(1)+\cdots+r(n))}\!\upharpoonright_{\prod_{i=1}^n\Gamma_R(\mathbf{X}_i\,;\,N,m,\delta)}  
\end{align}   
with $N \geq N_0$. When $N \geq N_0$, one has 
\begin{align*} 
&\chi_{\mathrm{orb},R}(\mathbf{X}_1,\dots,\mathbf{X}_n\,;\,N,m,\delta) 
\geq 
\chi_{\mathrm{orb},R}(\mathbf{X}_1,\dots,\mathbf{X}_n\,;\,N,m,\delta\,;\,\nu_R(N,m,\delta)) \\
&= 
\log\big((\gamma_{\mathrm{U}(N)}^{\otimes n}\otimes\Lambda_N^{\otimes(r(1)+\cdots+r(n))})(\Phi_N^{-1}(\Gamma_R(\mathbf{X}_1\sqcup\cdots\sqcup\mathbf{X}_n\,;\,N,m,\delta)))\big) \\
&\quad\quad\quad\quad\quad\quad\quad\quad\quad\quad\quad\quad\quad\quad\quad\quad\quad- 
\sum_{i=1}^n \log(\Lambda_N^{\otimes r(i)}(\Gamma_R(\mathbf{X}_i\,;\,N,m,\delta)),
\end{align*}
that is, 
\begin{align}\label{Eq6}
&\log\big((\gamma_{\mathrm{U}(N)}\otimes\Lambda_N^{\otimes(r(1)+\cdots+r(n))})(\Phi_N^{-1}(\Gamma_R(\mathbf{X}_1\sqcup\cdots\sqcup\mathbf{X}_n\,;\,N,m,\delta)))\big) \notag\\
&\quad\leq
\chi_{\mathrm{orb},R}(\mathbf{X}_1,\dots,\mathbf{X}_n\,;\,N,m,\delta) 
+ 
\sum_{i=1}^n \log(\Lambda_N^{\otimes r(i)}(\Gamma_R(\mathbf{X}_i\,;\,N,m,\delta))
\end{align}
for every $m \in \mathbb{N}$ and every $\delta > 0$. Let $(U_i)_{i=1}^n \in \mathrm{U}(N)^n$ be arbitrarily fixed, and then the $(U_i)_{i=1}^n$--section of $\Phi_N^{-1}(\Gamma_R(\mathbf{X}_1\sqcup\cdots\sqcup\mathbf{X}_n\,;\,N,m,\delta))$ defined to be 
$$
\{(\mathbf{A}_i)_{i=1}^n\,|\, \Phi_N((U_i)_{i=1}^n,(\mathbf{A}_i)_{i=1}^n) \in \Gamma_R(\mathbf{X}_1\sqcup\cdots\sqcup\mathbf{X}_n\,;\,N,m,\delta)\}
$$
clearly becomes $(U_i^*)_{i=1}^n\cdot\Gamma_R(\mathbf{X}_1\sqcup\cdots\sqcup\mathbf{X}_n\,;\,N,m,\delta)\cdot(U_i)_{i=1}^n$. Hence we have 
\begin{align*} 
&(\gamma_{\mathrm{U}(N)}^{\otimes n}\otimes\Lambda_N^{\otimes(r(1)+\cdots+r(n))})(\Phi_N^{-1}(\Gamma_R(\mathbf{X}_1\sqcup\cdots\sqcup\mathbf{X}_n\,;\,N,m,\delta))) \\
&\quad= 
\int_{\mathrm{U}(N)^n} 
\Lambda_N^{\otimes(r(1)+\cdots+r(n))}\big((U_i^*)_{i=1}^n\cdot\Gamma_R(\mathbf{X}_1\sqcup\cdots\sqcup\mathbf{X}_n\,;\,N,m,\delta)\cdot(U_i)_{i=1}^n\big)\,d\gamma_{\mathrm{U}(N)}^{\otimes n} \\
&\quad= 
\int_{\mathrm{U}(N)^n} 
\Lambda_N^{\otimes(r(1)+\cdots+r(n))}\big(\Gamma_R(\mathbf{X}_1\sqcup\cdots\sqcup\mathbf{X}_n\,;\,N,m,\delta)\big)\,d\gamma_{\mathrm{U}(N)}^{\otimes n} \\
&\quad= 
\Lambda_N^{\otimes(r(1)+\cdots+r(n))}\big(\Gamma_R(\mathbf{X}_1\sqcup\cdots\sqcup\mathbf{X}_n\,;\,N,m,\delta)\big)
\end{align*}
as in Lemma \ref{L-2-2}, since $(\mathbf{A}_i)_{i=1}^n \mapsto (U_i^*)_{i=1}\cdot(\mathbf{A}_i)_{i=1}^n\cdot(U_i)_{i=1}^n = (U_i^*\mathbf{A}_i U_i)_{i=1}^n$ induces an orthogonal transformation on $\prod_{1 \leq i \leq n}(M_N^\mathrm{sa})^{r(i)}$. By taking the limit as $m\to\infty$, $\delta\searrow 0$ after taking the limit superior as $N\to\infty$ of \eqref{Eq6} plus $(\sum_{i=1}^n r(i)/2)\log N$ we get the desired inequality thanks to Lemma \ref{L-2-5} and \cite[Proposition 2.4]{Voiculescu:InventMath94}.      
\end{proof} 

\begin{remark}\label{R-2-9} {\rm Here we assume that each $\mathbf{X}_i$ has f.d.a., that is, for every $m \in \mathbb{N}$ and $\delta > 0$ there is $N_{m,\delta} \in \mathbb{N}$ such that for every $N \geq N_{m,\delta}$ and every $1 \leq i \leq n$ one has $\Gamma_R(\mathbf{X}_i\,;\,N,m,\delta) \neq \emptyset$ and hence $\Lambda_N^{\otimes r(i)}(\Gamma_R(\mathbf{X}_i\,;\,N,m,\delta)) \neq 0$ with $R > \max\{\Vert X_{ij}\Vert_\infty\,|\,1 \leq j \leq r(i), 1 \leq i \leq n\}$. Thus we have the `uniform' probability measures $\nu_R(N,m,\delta)$ on $\prod_{1\leq i\leq n}\Gamma_R(\mathbf{X}_i\,;\,N,m,\delta)$ given as \eqref{Eq5} for every $N \geq N_{m,\delta}$. For every $N \geq N_{m,\delta}$ we have  
\begin{align*}
\chi_{\mathrm{orb},R}(\mathbf{X}_1,\dots,\mathbf{X}_n\,;\,N,m,\delta\,;\,\nu_R(N,m,\delta))& \\
= 
\log\big((\gamma_{\mathrm{U}(N)}^{\otimes n}\otimes\Lambda_N^{\otimes(r(1)+\cdots+r(n))})&(\Phi_N^{-1}(\Gamma_R(\mathbf{X}_1\sqcup\cdots\sqcup\mathbf{X}_n\,;\,N,m,\delta)))\big) \\
&\quad- 
\sum_{i=1}^n \log(\Lambda_N^{\otimes r(i)}(\Gamma_R(\mathbf{X}_i\,;\,N,m,\delta)),
\end{align*}
and hence
\begin{align*}
\chi_{\mathrm{orb},R}(\mathbf{X}_1,\dots,\mathbf{X}_n\,;\,N,m,\delta\,;\,\nu_R(N,m,\delta)) + \sum_{i=1}^n \log(\Lambda_N^{\otimes r(i)}(\Gamma_R(\mathbf{X}_i\,;\,N,m,\delta))& \notag\\
= 
\log\big(\Lambda_N^{\otimes(r(1)+\cdots+r(n))}\big(\Gamma_R(\mathbf{X}_1\sqcup\cdots\sqcup\mathbf{X}_n\,;\,N,m,\delta)\big)\big)&
\end{align*}
as in the proof of Proposition \ref{P-2-8}. 
This equality may be viewed as a microscopic version of the desired `equality'. Hence, if the limit superior as $N\to\infty$ in the definition of every $\chi(\mathbf{X}_i)$ could be replaced by the limit, then the quantity 
\begin{equation*}
C := \sup_{R>0}\lim_{\substack{m\to\infty \\ \delta \searrow 0}}\limsup_{N\to\infty} 
\frac{1}{N^2}\chi_{\mathrm{orb},R}(\mathbf{X}_1,\dots,\mathbf{X}_n\,;\,N,m,\delta\,;\,\nu_R(N,m,\delta))
\end{equation*}
would satisfy that $\chi(\mathbf{X}_1\sqcup\cdots\sqcup\mathbf{X}_n) = C + \sum_{i=1}^n \chi(\mathbf{X}_i)$. Moreover, its ultrafilter variant $C^\omega$ with replacing the limit superior as $N\to\infty$ by the limit as $N\to\omega$ clearly satisfies that $\chi^\omega(\mathbf{X}_1\sqcup\cdots\sqcup\mathbf{X}_n) = C^\omega + \sum_{i=1}^n \chi^\omega(\mathbf{X}_i)$. It would be nice if the above quantity $C$ and $\chi_\mathrm{orb}$ turned out to be the same quantity. However, we cannot say anything about this at the moment. }
\end{remark}

\section{Conditional Variant $\chi_\mathrm{orb}(\mathbf{X}_1,\dots,\mathbf{X}_n:\mathbf{v})$} As in the previous section, let $\mathbf{X}_i = (X_{i1},\dots,X_{ir(i)})$, $1 \leq i \leq n$ be arbitrary random multi-variables in $(\mathcal{M},\tau)$. Let $\mathbf{v} = (v_1,\dots,v_s)$ be an $s$-tuple of unitaries in $(\mathcal{M},\tau)$. The set of matricial microstates $\Gamma_R(\mathbf{X}_1\sqcup\cdots\sqcup\mathbf{X}_n,\mathbf{v}\,;\,N,m,\delta)$ is defined to be all the $((\mathbf{A}_i)_{i=1}^n,(V_i)_{i=1}^s) \in \prod_{i=1}^n ((M_N^\mathrm{sa})_R)^{r(i)} \times \mathrm{U}(N)^s$ such that 
\begin{align*}
|&\mathrm{tr}_N(h(A_{11},\dots,A_{1r(1)},\dots,A_{n1},\dots,A_{n r(n)},V_1,\dots,V_s)) \\
&\quad\quad- \tau(h(X_{11},\dots,X_{1r(1)},\dots,X_{n1},\dots,X_{n r(n)},v_1,\dots,v_s)))| < \delta
\end{align*}
for every $*$-monomial $h$ in $(r(1)+\cdots+r(n)+s)$ indeterminates of degree not greater than $m$. Denote by $\Gamma_R(\mathbf{X}_1\sqcup\cdots\sqcup\mathbf{X}_n:\mathbf{v}\,;\,N,m,\delta)$ the set of all $(\mathbf{A}_i)_{i=1}^n \in \prod_{i=1}^n ((M_N^\mathrm{sa})_R)^{r(i)}$ such that $((\mathbf{A}_i)_{i=1}^n,(V_i)_{i=1}^s) \in \Gamma_R(\mathbf{X}_1\sqcup\cdots\sqcup\mathbf{X}_n,\mathbf{v}\,;\,N,m,\delta)$ for some $(V_i)_{i=1}^s \in \mathrm{U}(N)^s$. For given multi-matrices $\mathbf{A}_i = (A_{ij})_{j=1}^{r(i)} \in (M_N^\mathrm{sa})^{r(i)}$, $1 \leq i \leq n$, the set of orbital microstates $\Gamma_\mathrm{orb}(\mathbf{X}_1,\dots,\mathbf{X}_n:(\mathbf{A}_i)_{i=1}^n:\mathbf{v}\,;\,N,m,\delta)$ \emph{in presence of $\mathbf{v}$} is defined to be all $(U_i)_{i=1}^n \in \mathrm{U}(N)^n$ such that $(U_i\mathbf{A}_i U_i^*)_{i=1}^n \in \Gamma_\infty(\mathbf{X}_1\sqcup\cdots\sqcup\mathbf{X}_n:\mathbf{v}\,;\,N,m,\delta)$. With these notations the \emph{orbital free entropy $\chi_\mathrm{orb}(\mathbf{X}_1,\dots,\mathbf{X}_n:\mathbf{v})$ of $\mathbf{X}_1,\dots,\mathbf{X}_n$ in presence of $\mathbf{v}$} is defined in the exactly same way as in Definition \ref{D-2-1} with replacing $\Gamma_R(\mathbf{X}_1,\dots,\mathbf{X}_n\,;\,N,m,\delta)$ by $\Gamma_R(\mathbf{X}_1,\dots,\mathbf{X}_n:\mathbf{v}\,;\,N,m,\delta)$. Clearly Lemma \ref{L-2-2} still holds in this setting; namely one has 
\begin{align*}
&(\gamma_{\mathrm{U}(N)}^{\otimes n}\otimes\mu)(\Phi_N^{-1}(\Gamma_R(\mathbf{X}_1\sqcup\cdots\sqcup\mathbf{X}_n:\mathbf{v}\,;\,N,m,\delta))) \\ 
&= 
\int_{\prod_{i=1}^n ((M_N^\mathrm{sa})_R)^{r(i)}} \gamma_{\mathrm{U}(N)}^{\otimes n}\big(\Gamma_\mathrm{orb}(\mathbf{X}_1,\dots,\mathbf{X}_n:(\mathbf{A}_i)_{i=1}^n:\mathbf{v}\,;\,N,m,\delta)\big)\,d\mu \\
&= 
\int_{\prod_{i=1}^n \Gamma_R(\mathbf{X}_i\,;\,N,m,\delta)} \gamma_{\mathrm{U}(N)}^{\otimes n}\big(\Gamma_\mathrm{orb}(\mathbf{X}_1,\dots,\mathbf{X}_n:(\mathbf{A}_i)_{i=1}^n:\mathbf{v}\,;\,N,m,\delta)\big)\,d\mu  
\end{align*} 
holds for every probability measure $\mu \in \mathcal{P}(\prod_{i=1}^n(M_N^\mathrm{sa})^{r(i)})$. Hence we can prove the same assertions as Proposition \ref{P-2-3}, Proposition \ref{P-2-4} and Lemma \ref{L-2-5} even for $\chi_\mathrm{orb}(\mathbf{X}_1,\dots,\mathbf{X}_n:\mathbf{v})$ by the same arguments there with obvious modifications (e.g.~replacing $\Gamma_\mathrm{orb}(\mathbf{X}_1,\dots,\mathbf{X}_n:(\mathbf{A}_i)_{i=1}^n\,;\,N,m,\delta)$ there by $\Gamma_\mathrm{orb}(\mathbf{X}_1,\dots,\mathbf{X}_n:(\mathbf{A}_i)_{i=1}^n:\mathbf{v}\,;\,N,m,\delta)$). In particular, the present definition of $\chi_\mathrm{orb}(\mathbf{X}_1,\dots,\mathbf{X}_n:\mathbf{v})$ completely agrees with the previous one in \cite{HiaiMiyamotoUeda:IJM09}. 

The variant $\chi_\mathrm{orb}(\mathbf{X}_1,\dots,\mathbf{X}_n:\mathbf{v})$ is necessary in the next section to define the (modified) orbital free entropy dimension $\delta_{0,\mathrm{orb}}(\mathbf{X},\dots,\mathbf{X}_n)$. For the purpose we provide the next two facts, which generalize the previous ones \cite[Propositions 4.6, 4.7]{HiaiMiyamotoUeda:IJM09} to arbitrary random multi-variables. Note that the first one may be regarded as the $\chi_\mathrm{orb}$-counterpart of \cite[Proposition 10.4]{Voiculescu:AdvMath99}. 

\begin{theorem}\label{T-3-1} Let $\mathbf{v} = (v_1,\dots,v_n)$ be a freely independent $n$-tuple of unitary random variables in $(\mathcal{M},\tau)$ with $\chi_u(v_i) > -\infty$ for all $1 \leq i \leq n$, where $\chi_u(-)$ means free entropy of unitary random variables, see \cite[\S6.5]{HiaiPetz:Book}. If $\mathbf{X}_1\sqcup\cdots\sqcup\mathbf{X}_n$ are freely independent of $\mathbf{v}$ in $(\mathcal{M},\tau)$, then 
\begin{align*} 
\chi_\mathrm{orb}(\mathbf{X}_1,\dots,\mathbf{X}_n) 
\leq 
\chi_\mathrm{orb}(v_1\mathbf{X}_1 v_1^*,\dots,v_n\mathbf{X}_n v_n^*:\mathbf{v}) 
\leq 
\chi_\mathrm{orb}(v_1\mathbf{X}_1 v_1^*,\dots,v_n\mathbf{X}_n v_n^*). 
\end{align*} 
\end{theorem}
\begin{proof} The second inequality is trivial; hence it suffices to prove the first. 

 Set $R := \max\{\Vert X_{ij}\Vert_\infty\,|\, 1 \leq j \leq r(i), 1 \leq i \leq n\}$. Let $m \in \mathbb{N}$ and $\delta > 0$ be arbitrary. We can choose $\delta' > 0$ in such a way that for every $N \in \mathbb{N}$ we have: If $(\mathbf{A}_i)_{i=1}^n \in \Gamma_R(\mathbf{X}_1\sqcup\cdots\sqcup\mathbf{X}_n\,;\,N,m,\delta')$ and $(V_i)_{i=1}^n \in \Gamma(\mathbf{v}\,;\,N,2m,\delta')$ are $(3m,\delta')$-free, then $((V_i\mathbf{A}_i V_i^*)_{i=1}^n,(V_i)_{i=1}^n)$ falls in $\Gamma_R(v_1\mathbf{X}_1 v_1^*\sqcup\cdots\sqcup v_n\mathbf{X}_n v_n^*,\mathbf{v}\,;\,N,m,\delta)$. Here $\Gamma(\mathbf{v}\,;\,N,2m,\delta')$ denotes all the $n$-tuples of $N\times N$ unitary matrices $(V_i)_{i=1}^n$ such that 
$$
|\mathrm{tr}_N(V_{i_1}^{\varepsilon_1}\cdots V_{i_l}^{\varepsilon_l}) - \tau(v_{i_1}^{\varepsilon_1}\cdots v_{i_l}^{\varepsilon_l})| < \delta'
$$
whenever $1 \leq i_1,\dots,i_l \leq n$, $\varepsilon_1,\dots,\varepsilon_l \in \{1,*\}$ and $1 \leq l \leq 2m$.  

We may and do assume that $\chi_{\mathrm{orb},R}(\mathbf{X}_1,\dots,\mathbf{X}_n) = \chi_\mathrm{orb}(\mathbf{X}_1,\dots,\mathbf{X}_n) > -\infty$ (due to Corollary \ref{C-2-7}). Then there is a subsequence $N_1 < N_2 < \cdots$ so that $\bar{\chi}_{\mathrm{orb},R}(\mathbf{X}_1,\dots,\mathbf{X}_n\,;\,N_k,m,\delta') > -\infty$ for all $k \in \mathbb{N}$ and 
\begin{align*} 
\limsup_{N\to\infty}\frac{1}{N^2}\bar{\chi}_{\mathrm{orb},R}(\mathbf{X}_1,\dots,\mathbf{X}_n\,;\,N,m,\delta') = 
\lim_{k\to\infty}\frac{1}{N_k^2}\bar{\chi}_{\mathrm{orb},R}(\mathbf{X}_1,\dots,\mathbf{X}_n\,;\,N_k,m,\delta'). 
\end{align*}
For each $k \in \mathbb{N}$ one can choose $\mathbf{A}_i^{(k)} \in ((M_N^\mathrm{sa})_R)^{r(i)}$, $1 \leq i \leq n$, in such a way that 
\begin{align*} 
\bar{\chi}_{\mathrm{orb},R}(\mathbf{X}_1,\dots,\mathbf{X}_n\,;\,N_k,m,\delta') - 1 
\leq 
\log\Big(\gamma_{\mathrm{U}(N)}^{\otimes n}\big(\Gamma_\mathrm{orb}(\mathbf{X}_1,\dots,\mathbf{X}_n:(\mathbf{A}_i^{(k)})_{i=1}^n\,;\,N_k,m,\delta')\big)\Big). 
\end{align*} 
Then the exactly same argument as in the proof of \cite[Proposition 4.6]{HiaiMiyamotoUeda:IJM09} with replacing $N$, $\Xi_i(N)$'s, $\rho$ and $\delta$ there by $N_k$, $\mathbf{A}_i^{(k)}$'s, $\delta$ and $\delta'$, respectively, shows that 
\begin{align*}
&\frac{1}{2}\gamma_{\mathrm{U}(N)}^{\otimes n}\big(\Gamma_\mathrm{orb}(\mathbf{X}_1,\dots,\mathbf{X}_n:(\mathbf{A}_i^{(k)})_{i=1}^n\,;\,N_k,m,\delta')\big) \\
&\quad\quad\leq 
\gamma_{\mathrm{U}(N)}^{\otimes n}\big(\Gamma_\mathrm{orb}(v_1\mathbf{X}_1 v_1^*,\dots,v_n\mathbf{X}_n v_n^*:(\mathbf{A}_i^{(k)})_{i=1}^n:\mathbf{v}\,;\,N_k,m,\delta)\big)
\end{align*}
for all sufficiently large $k \in \mathbb{N}$. Therefore, we have 
\begin{align*}
&\chi_\mathrm{orb}(\mathbf{X}_1,\dots,\mathbf{X}_n) = \chi_{\mathrm{orb},R}(\mathbf{X}_1,\dots,\mathbf{X}_n) \quad \text{(by Corollary \ref{C-2-7})}\\
&\leq 
\lim_{k\to\infty}\frac{1}{N_k^2}\bar{\chi}_{\mathrm{orb},R}(\mathbf{X}_1,\dots,\mathbf{X}_n\,;\,N_k,m,\delta') \\
&\leq 
\limsup_{k\to\infty}\frac{1}{N_k^2}\Big[\log\Big(\gamma_{\mathrm{U}(N)}^{\otimes n}\big(\Gamma_\mathrm{orb}(v_1\mathbf{X}_1 v_1^*,\dots,v_n\mathbf{X}_n v_n^*:(\mathbf{A}_i^{(k)})_{i=1}^n:\mathbf{v}\,;\,N_k,m,\delta)\big)\Big)+\log2 + 1\Big] \\
&\leq 
\limsup_{N\to\infty}\frac{1}{N^2}\bar{\chi}_{\mathrm{orb},R}(v_1\mathbf{X}_1 v_1^*,\dots,v_n\mathbf{X}_n v_n^*:\mathbf{v}\,;\,N,m,\delta), 
\end{align*}
from which the desired assertion immediately follows.  
\end{proof}    

\begin{proposition}\label{P-3-2} Let $\mathbf{v} = (v_1,\dots,v_n)$ be a freely independent $n$-tuple of unitary random variables in $(\mathcal{M},\tau)$. If $\mathbf{X}_1\sqcup(v_1)$ is freely independent of $\mathbf{X}_2\sqcup\cdots\sqcup\mathbf{X}_n\sqcup(v_2,\dots,v_n)$ in $(\mathcal{M},\tau)$, then 
$$
\chi_\mathrm{orb}(\mathbf{X}_1,\dots,\mathbf{X}_n:\mathbf{v}) = 
\chi_\mathrm{orb}(\mathbf{X}_1,:v_1) + \chi_\mathrm{orb}(\mathbf{X}_2,\dots,\mathbf{X}_n:(v_2,\dots,v_n)) 
$$
whenever $(\mathbf{X}_1,v_1)$ is regular, which means that 
$$
\chi_\mathrm{orb}(\mathbf{X}_1:v_1) = \chi_{\mathrm{orb},R}(\mathbf{X}_1:v_1) = \lim_{\substack{m\to\infty \\ \delta \searrow 0}}\liminf_{N\to\infty}\frac{1}{N^2}\bar{\chi}_{\mathrm{orb},R}(\mathbf{X}_1:v_1\,;\,N,m,\delta)
$$
holds as long as $R > \max\{\Vert X_{1j}\Vert_\infty\,|\,1\leq j \leq r(1)\}$.  
\end{proposition}
\begin{proof} Since inequality `$\leq$' in the desired assertion trivially holds, it suffices to prove the reverse `$\geq$' under the assumption that both $\chi_\mathrm{orb}(\mathbf{X}_1:v_1) > -\infty$ and $\chi_\mathrm{orb}(\mathbf{X}_2,\dots,\mathbf{X}_n:(v_2,\dots,v_n)) > -\infty$; otherwise it is trivial. 

Let $R > \max\{\Vert X_{ij}\Vert_\infty\,|\,1\leq j \leq r(i), 1 \leq i \leq n\}$ be fixed. Let $m \in \mathbb{N}$ and $\delta > 0$ be arbitrary. We can choose $\delta' > 0$ in such a way that for every $N \in \mathbb{N}$ we have: If $(\mathbf{A}_1,V_1) \in \Gamma_R(\mathbf{X}_1, v_1\,;\,N,m,\delta')$ and $((\mathbf{A}_i)_{i=2}^n,(V_i)_{i=2}^n) \in \Gamma_R(\mathbf{X}_2\sqcup\cdots\sqcup\mathbf{X}_n,(v_2,\dots,v_n)\,;\,N,m,\delta')$ are $(m,\delta')$-free, then $((\mathbf{A}_i)_{i=1}^n,(V_i)_{i=1}^n)$ falls in $\Gamma_R(\mathbf{X}_1\sqcup\cdots\sqcup\mathbf{X}_n,\mathbf{v}\,;\,N,m,\delta)$.  

Since $\chi_{\mathrm{orb},R}(\mathbf{X}_2,\dots,\mathbf{X}_n:(v_2,\dots,v_n)) = \chi_\mathrm{orb}(\mathbf{X}_2,\dots,\mathbf{X}_n:(v_2,\dots,v_n)) > -\infty$ (due to the $\chi_\mathrm{orb}(\,-:\mathbf{v})$-counterpart of Lemma \ref{L-2-5}), one can choose a subsequence $N_1 < N_2 < \cdots$ in such a way that $\bar{\chi}_{\mathrm{orb},R}(\mathbf{X}_2,\dots,\mathbf{X}_n:(v_2,\dots,v_n)\,;\,N_k,m,\delta') > -\infty$ for all $k \in \mathbb{N}$ and 
\begin{align*} 
&\limsup_{N\to\infty}\frac{1}{N^2}\bar{\chi}_{\mathrm{orb},R}(\mathbf{X}_2,\dots,\mathbf{X}_n:(v_2,\dots,v_n)\,;\,N,m,\delta') \\
&\quad\quad\quad= 
\lim_{k\to\infty}\frac{1}{N_k^2}\bar{\chi}_{\mathrm{orb},R}(\mathbf{X}_2,\dots,\mathbf{X}_n:(v_2,\dots,v_n)\,;\,N_k,m,\delta'). 
\end{align*}
Taking a subsequence of $N_k$ if necessary we may and do assume that $\bar{\chi}_{\mathrm{orb},R}(\mathbf{X}_1:v_1\,;\,N_k,m,\delta') > -\infty$ for all $k \in \mathbb{N}$, since $(\mathbf{X}_1,v_1)$ is regular. Then one can choose $\mathbf{A}_i^{(k)} \in ((M_N^\mathrm{sa})_R)^{r(i)}$, $1 \leq i \leq n$, in such a way that 
\begin{align*}
&\bar{\chi}_{\mathrm{orb},R}(\mathbf{X}_1:v_1\,;\,N_k,m,\delta') - 1 < 
\chi_\mathrm{orb}(\mathbf{X}_1:\mathbf{A}_1^{(k)}:v_1\,;\,N_k,m,\delta'), \\
&\bar{\chi}_{\mathrm{orb},R}(\mathbf{X}_2,\dots,\mathbf{X}_n:(v_2,\dots,v_n)\,;\,N_k,m,\delta') - 1 \\
&\quad\quad\quad< 
\chi_\mathrm{orb}(\mathbf{X}_2,\dots,\mathbf{X}_n:(\mathbf{A}_i^{(k)})_{i=2}^n:(v_2,\dots,v_n)\,;\,N_k,m,\delta').  
\end{align*} 
With letting
\begin{align*} 
\Phi(N_k,m,\delta') 
&:= 
\Gamma_\mathrm{orb}(\mathbf{X}_1:\mathbf{A}_1^{(k)}:v_1\,;\,N_k,m,\delta') \\
&\quad\quad\quad\times
\Gamma_\mathrm{orb}(\mathbf{X}_2,\dots,\mathbf{X}_n:(\mathbf{A}_i^{(k)})_{i=2}^n:(v_2,\dots,v_n)\,;\,N_k,m,\delta'), \\
\Psi(N_k,m,\delta) 
&:= 
\Gamma_\mathrm{orb}(\mathbf{X}_1,\dots,\mathbf{X}_n:(\mathbf{A}_i^{(k)})_{i=1}^n:\mathbf{v}\,;\,N_k,m,\delta)
\end{align*}
the exactly same argument as in the proof of \cite[Eq.(4.7)]{HiaiMiyamotoUeda:IJM09} with replacing $N$, $\Xi_i(N)$'s, $\rho$ and $\delta$ there by $N_k$, $\mathbf{A}_i^{(k)}$'s, $\delta$ and $\delta'$, respectively,  shows that 
$$
\frac{\gamma_{\mathrm{U}(N_k)}^{\otimes n}(\Psi(N_k,m,\delta)\cap\Phi(N_k,m,\delta'))}{\gamma_{\mathrm{U}(N_k)}^{\otimes n}(\Phi(N_k,m,\delta'))} \geq \frac{1}{2} 
$$ 
for all sufficiently large $k \in \mathbb{N}$. Therefore, we have
\begin{align*}
&\limsup_{N\to\infty}\frac{1}{N^2}\bar{\chi}_{\mathrm{orb},R}(\mathbf{X}_1,\dots,\mathbf{X}_n:\mathbf{v}\,;\,N,m,\delta) \\
&\geq 
\limsup_{k\to\infty}\frac{1}{N_k^2}\bar{\chi}_{\mathrm{orb},R}(\mathbf{X}_1,\dots,\mathbf{X}_n:\mathbf{v}\,;\,N_k,m,\delta) \\
&\geq 
\limsup_{k\to\infty}\frac{1}{N_k^2}\log\Big(\gamma_{\mathrm{U}(N_k)}^{\otimes n}\big(\Psi(N_k,m,\delta)\big)\Big) \\
&\geq 
\limsup_{k\to\infty}\frac{1}{N_k^2}\Big[\log\Big(\gamma_{\mathrm{U}(N_k)}^{\otimes n}\big(\Phi(N_k,m,\delta')\big)\Big)-\log2\Big] \\
&\geq 
\liminf_{k\to\infty} \frac{1}{N_k^2}\log\Big(\gamma_{\mathrm{U}(N_k)}\big(\Gamma_\mathrm{orb}(\mathbf{X}_1:\mathbf{A}_1^{(k)}:v_1\,;\,N_k,m,\delta')\big)\Big) \\
&\quad\quad+
\lim_{k\to\infty}\frac{1}{N_k^2}\log\Big(\gamma_{\mathrm{U}(N_k)}^{\otimes n-1}\big(\Gamma_\mathrm{orb}(\mathbf{X}_2,\dots,\mathbf{X}_n:(\mathbf{A}_i^{(k)})_{i=2}^n:(v_2,\dots,v_n)\,;\,N_k,m,\delta')\big)\Big) \\
&\geq 
\liminf_{k\to\infty} \frac{1}{N_k^2}\Big(\bar{\chi}_{\mathrm{orb},R}(\mathbf{X}_1:v_1\,;\,N_k,m,\delta')\big) - 1\Big) \\
&\quad\quad+
\lim_{k\to\infty}\frac{1}{N_k^2}\Big(\bar{\chi}_{\mathrm{orb},R}(\mathbf{X}_2,\dots,\mathbf{X}_n:(v_2,\dots,v_n)\,;\,N_k,m,\delta') - 1\Big) \\
&\geq 
\liminf_{N\to\infty} \frac{1}{N^2}\bar{\chi}_{\mathrm{orb},R}(\mathbf{X}_1:v_1\,;\,N,m,\delta')\big) \\
&\quad\quad+
\limsup_{N\to\infty}\frac{1}{N^2}\bar{\chi}_{\mathrm{orb},R}(\mathbf{X}_2,\dots,\mathbf{X}_n:(v_2,\dots,v_n)\,;\,N,m,\delta') \\
&\geq 
\chi_\mathrm{orb}(\mathbf{X}_1:v_1) + \chi_\mathrm{orb}(\mathbf{X}_2,\dots,\mathbf{X}_n:(v_2,\dots,v_n)),
\end{align*}
where the last inequality follows from the regularity assumption on $(\mathbf{X}_1,v_1)$ together with the $\chi_\mathrm{orb}(\,-:\mathbf{v})$-counterpart of Lemma \ref{L-2-5}. Hence we are done. 
\end{proof}

\section{Orbital Free Entropy Dimension $\delta_{0,\mathrm{orb}}(\mathbf{X}_1,\dots,\mathbf{X}_n)$} 

With the materials in the previous section we can generalize orbital free entropy dimension $\delta_{0,\mathrm{orb}}$ to arbitrary random multi-variables $\mathbf{X}_i = (X_{i1},\dots,X_{ir(i)})$, $1 \leq i \leq n$, in $(\mathcal{M},\tau)$.

\begin{definition} \label{D-4-1} Let $\mathbf{v}(t) = (v_1(t),\dots,v_n(t))$, $t \geq 0$, be a freely independent $n$-tuple of free unitary multiplicative Brownian motions in $(\mathcal{M},\tau)$ {\rm(}see \cite{Biane:FieldsInstCommun97}{\rm)} starting at $\mathbf{1} = (1,\dots,1)$, which are chosen to be freely independent of $\mathbf{X}_1\sqcup\cdots\sqcup\mathbf{X}_n$ in $(\mathcal{M},\tau)$. The {\rm(}modified{\rm)} orbital free entropy dimension of $\mathbf{X}_1,\dots,\mathbf{X}_n$ is defined to be 
\begin{align*}
\delta_{0,\mathrm{orb}}(\mathbf{X}_1,\dots,\mathbf{X}_n) := 
\limsup_{\varepsilon\searrow0}\frac{\chi_\mathrm{orb}(v_1(\varepsilon)\mathbf{X}_1 v_1(\varepsilon)^*,\dots,v_n(\varepsilon)\mathbf{X}_1 v_n(\varepsilon)^*:\mathbf{v}(\varepsilon))}{|\log\sqrt{\varepsilon}|}. 
\end{align*} 
\end{definition}

We need a simple lemma. 

\begin{lemma}\label{L-4-1} Let $\mathbf{X} = (X_1,\dots,X_r)$ be a random multi-variable that has f.d.a., and $v$ be a unitary random variable with $\chi_u(v) > -\infty$. Then $(\mathbf{X},v)$ is regular in the sense of Proposition \ref{P-3-2}, and moreover $\chi_\mathrm{orb}(\mathbf{X}:v) = 0$.  
\end{lemma}
\begin{proof} Let $R > \max\{\Vert X_j\Vert_\infty\,|\,1\leq j \leq r\}$ be fixed, and $m \in \mathbb{N}, \delta > 0$ be arbitrary.  The exactly same argument as in the proof of \cite[Proposition 4.6] {HiaiMiyamotoUeda:IJM09} shows that there is $\delta' > 0$ such that for all sufficiently large $N \in \mathbb{N}$ one can choose $\mathbf{A}^{(N)} \in \Gamma_R(\mathbf{X}\,;\,N,m,\delta)$, and then   
$$
\frac{1}{2}\gamma_{\mathrm{U}(N)}\big(\Gamma_\mathrm{orb}(\mathbf{X}:\mathbf{A}^{(N)}\,;\,N,m,\delta')\big) \leq 
\gamma_{\mathrm{U}(N)}\big(\Gamma_\mathrm{orb}(v\mathbf{X}v^*:\mathbf{A}^{(N)}:v\,;\,N,m,\delta)\big)
$$
holds. Clearly $\Gamma_\mathrm{orb}(\mathbf{X}:\mathbf{A}^{(N)}\,;\,N,m,\delta') = \mathrm{U}(N)$, and the desired assertion is immediate. \end{proof}  

Here are fundamental properties of $\delta_{0,\mathrm{orb}}$. With the above lemma the exactly same pattern of arguments as in \cite[Proposition 5.3]{HiaiMiyamotoUeda:IJM09} works well for showing the next proposition.  

\begin{proposition}\label{P-4-2} We have\,{\rm:} 
\begin{itemize}
\item[(1)] $\delta_{0,\mathrm{orb}}(\mathbf{X}_1,\dots,\mathbf{X}_n) \leq 0$. 
\item[(2)] $\delta_{0,\mathrm{orb}}(\mathbf{X}) = 0$ if $\mathbf{X}$ has f.d.a. 
\item[(3)] $\delta_{0,\mathrm{orb}}(\mathbf{X}_1,\dots,\mathbf{X}_{n'},\mathbf{X}_{n'+1},\dots,\mathbf{X}_n) \leq \delta_{0,\mathrm{orb}}(\mathbf{X}_1,\dots,\mathbf{X}_{n'}) + \delta_{0,\mathrm{orb}}(\mathbf{X}_{n'+1},\dots,\mathbf{X}_n)$. 
\item[(4)] If $\mathbf{Y}_i = (Y_{ij})_{j=1}^{r'(i)} \subseteq W^*(\mathbf{X}_n)$, $1 \leq i \leq n$, then $\delta_{0,\mathrm{orb}}(\mathbf{X}_1,\dots,\mathbf{X}_n) \leq \delta_{0,\mathrm{orb}}(\mathbf{Y}_1,\dots,\mathbf{Y}_n)$. In particular, $\delta_{0,\mathrm{orb}}(\mathbf{X}_1,\dots,\mathbf{X}_n)$ depends only on $W^*(\mathbf{X}_1), \dots, W^*(\mathbf{X}_n)$. 
\item[(5)] If $\chi_\mathrm{orb}(\mathbf{X}_1,\dots,\mathbf{X}_n) > -\infty$, then $\delta_{0,\mathrm{orb}}(\mathbf{X}_1,\dots,\mathbf{X}_n) = 0$. This is the case when the $\mathbf{X}_i$'s have f.d.a.~and are freely independent in $(\mathcal{M},\tau)$. 
\item[(6)] If $\mathbf{X}_1$ and $\mathbf{X}_2\sqcup\cdots\sqcup\mathbf{X}_n$ are freely independent in $(\mathcal{M},\tau)$ and if $\mathbf{X}_1$ has f.d.a., then $\delta_{0,\mathrm{orb}}(\mathbf{X}_1,\mathbf{X}_2,\dots,\mathbf{X}_n) = \delta_{0,\mathrm{orb}}(\mathbf{X}_2,\dots,\mathbf{X}_n)$. 
\end{itemize} 
\end{proposition}
\begin{proof} (1) and (3) are trivial. (2) follows from Lemma \ref{L-4-1}. (4) needs only that the $\chi_\mathrm{orb}(-:\mathbf{v})$-counterpart of Theorem \ref{T-2-6} (6) be proved; the same proof with obvious modifications works to do so. (5) follows from Theorem \ref{T-3-1} since $\chi_u(v_i(\varepsilon)) > -\infty$ for every $\varepsilon > 0$. Finally (6) follows from Proposition \ref{P-3-2} with the help of Lemma \ref{L-4-1}.  
\end{proof}  

It is natural to examine how Jung's covering/packing approach to $\delta_0$ \cite{Jung:PacificJMath03} works for $\delta_{0,\mathrm{orb}}$ as in \cite[\S5]{HiaiMiyamotoUeda:IJM09} without the hyperfiniteness assumption. Here is a new definition of $\delta_{1,\mathrm{orb}}$. (See \cite[\S5]{HiaiMiyamotoUeda:IJM09} for a brief explanation on covering/packing numbers in metric spaces.)  

\begin{definition}\label{D-4-3} For each $R > 0$ possibly with $R = \infty$, each $N, m \in \mathbb{N}$ and $\delta, \varepsilon > 0$ we define 
\begin{align*} 
\mathbb{K}_{\varepsilon,R}^\mathrm{orb}(\mathbf{X}_1,\dots,\mathbf{X}_n\,;\,&N,m,\delta) \\
&:= \sup_{\mathbf{A}_i \in ((M_N^\mathrm{sa})_R)^{r(i)}} \log\Big(K_\varepsilon\big(\Gamma_\mathrm{orb}(\mathbf{X}_1,\dots,\mathbf{X}_n:(\mathbf{A}_i)_{i=1}^n\,;\,N,m,\delta)\big)\Big) \\
&:= \sup_{\mathbf{A}_i \in \Gamma_R(\mathbf{X}_i\,;\,N,m,\delta)} \log\Big(K_\varepsilon\big(\Gamma_\mathrm{orb}(\mathbf{X}_1,\dots,\mathbf{X}_n:(\mathbf{A}_i)_{i=1}^n\,;\,N,m,\delta)\big)\Big), \\
\mathbb{P}_{\varepsilon,R}^\mathrm{orb}(\mathbf{X}_1,\dots,\mathbf{X}_n\,;\,&N,m,\delta) \\
&:= \sup_{\mathbf{A}_i \in ((M_N^\mathrm{sa})_R)^{r(i)}} \log\Big(P_\varepsilon\big(\Gamma_\mathrm{orb}(\mathbf{X}_1,\dots,\mathbf{X}_n:(\mathbf{A}_i)_{i=1}^n\,;\,N,m,\delta)\big)\Big) \\
&:= \sup_{\mathbf{A}_i \in \Gamma_R(\mathbf{X}_i\,;\,N,m,\delta)} \log\Big(P_\varepsilon\big(\Gamma_\mathrm{orb}(\mathbf{X}_1,\dots,\mathbf{X}_n:(\mathbf{A}_i)_{i=1}^n\,;\,N,m,\delta)\big)\Big), \\
\intertext{{\rm(}with the same convention as in Proposition \ref{P-2-4}{\rm)}}
\mathbb{K}_{\varepsilon,R}^\mathrm{orb}(\mathbf{X}_1,\dots,\mathbf{X}_n) 
&:= 
\lim_{\substack{m\to\infty \\ \delta \searrow 0}}\limsup_{N\to\infty}\frac{1}{N^2}\mathbb{K}_{\varepsilon,R}^\mathrm{orb}(\mathbf{X}_1,\dots,\mathbf{X}_n\,;\,N,m,\delta), \\
\mathbb{P}_{\varepsilon,R}^\mathrm{orb}(\mathbf{X}_1,\dots,\mathbf{X}_n) 
&:= 
\lim_{\substack{m\to\infty \\ \delta \searrow 0}}\limsup_{N\to\infty}\frac{1}{N^2}\mathbb{P}_{\varepsilon,R}^\mathrm{orb}(\mathbf{X}_1,\dots,\mathbf{X}_n\,;\,N,m,\delta), \\
\mathbb{K}_\varepsilon^\mathrm{orb}(\mathbf{X}_1,\dots,\mathbf{X}_n) 
&:= \sup_{0 < R < \infty}\mathbb{K}_{\varepsilon,R}^\mathrm{orb}(\mathbf{X}_1,\dots,\mathbf{X}_n), \\
\mathbb{P}_\varepsilon^\mathrm{orb}(\mathbf{X}_1,\dots,\mathbf{X}_n) 
&:= \sup_{0 < R < \infty}\mathbb{P}_{\varepsilon,R}^\mathrm{orb}(\mathbf{X}_1,\dots,\mathbf{X}_n). 
\end{align*}
Here $K_\varepsilon(\Gamma)$ and $P_\varepsilon(\Gamma)$ for a subset $\Gamma$ in the metric space $\mathrm{U}(N)^n$ equipped with the metric $d_2((U_i)_{i=1}^n,(V_i)_{i=1}^n) := \sqrt{\sum_{i=1}^n \Vert U_i - V_i\Vert_{\mathrm{tr}_N,2}^2}$ denote the minimal number of $\varepsilon$-balls that cover $\Gamma$ and the maximal number of disjoint $\varepsilon$-balls inside $\Gamma$, respectively. {\rm(}Note that $P_\varepsilon(\Gamma) \geq K_{2\varepsilon}(\Gamma) \geq P_{4\varepsilon}(\Gamma)$ holds in general.{\rm)} Then we define 
$$
\delta_{1,\mathrm{orb}}(\mathbf{X}_1,\dots,\mathbf{X}_n) 
:= 
\limsup_{\varepsilon\searrow 0}\frac{\mathbb{K}^\mathrm{orb}_\varepsilon(\mathbf{X}_1,\dots,\mathbf{X}_n)}{|\log\varepsilon|} - n 
= \limsup_{\varepsilon\searrow 0}\frac{\mathbb{P}^\mathrm{orb}_\varepsilon(\mathbf{X}_1,\dots,\mathbf{X}_n)}{|\log\varepsilon|} - n.
$$
\end{definition} 

The next lemma can be shown in the exactly same way as in the proof of Lemma \ref{L-2-5}. 

\begin{lemma}\label{L-4-5} If $R > \max\{\Vert X_{ij}\Vert_\infty\,|\,1 \leq j \leq r(i), 1 \leq i \leq n\}$ possibly with $R = \infty$, then $\mathbb{K}_\varepsilon^\mathrm{orb}(\mathbf{X}_1,\dots,\mathbf{X}_n) = \mathbb{K}_{\varepsilon,R}^\mathrm{orb}(\mathbf{X}_1,\dots,\mathbf{X}_n)$ and $\mathbb{P}_\varepsilon^\mathrm{orb}(\mathbf{X}_1,\dots,\mathbf{X}_n) = \mathbb{P}_{\varepsilon,R}^\mathrm{orb}(\mathbf{X}_1,\dots,\mathbf{X}_n)$ for every $\varepsilon > 0$.  
\end{lemma}

The next proposition shows that Jung's approach still works well for $\delta_{0,\mathrm{orb}}(\mathbf{X}_1,\dots,\mathbf{X}_n)$ of arbitrary random multi-variables $\mathbf{X}_1,\dots,\mathbf{X}_n$.   

\begin{proposition}\label{P-4-6} We have $\delta_{0,\mathrm{orb}}(\mathbf{X}_1,\dots,\mathbf{X}_n) = \delta_{1,\mathrm{orb}}(\mathbf{X}_1,\dots,\mathbf{X}_n)$. 
\end{proposition} 
\begin{proof} Let $R > \max\{\Vert X_{ij}\Vert_\infty\,|\,1 \leq j \leq r(i), 1 \leq i \leq n\}$ be fixed throughout. 

By \cite[Lemma 8]{Biane:FieldsInstCommun97} there is $K > 0$ so that $\Vert v_i(t) - 1\Vert_\infty \leq K\sqrt{t}$ for every $0 \leq t \leq 1$. Set $C := K^2 + 2$. 

Let us first prove inequality `$\geq$' in the desired assertion. 
Let $m \in \mathbb{N}$ and $\delta, \varepsilon > 0$ be arbitrary with $\varepsilon \leq 1$. Then we can choose $0 < \delta' < \min\{\varepsilon,\delta\}$ in such a way that for every $N \in \mathbb{N}$ we have: If $(\mathbf{A}_i)_{i=1}^n \in \Gamma_R(\mathbf{X}_1\sqcup\cdots\sqcup\mathbf{X}_n\,;\,N,m,\delta')$ and $(V_i)_{i=1}^n \in \Gamma(\mathbf{v}(\varepsilon)\,;\,N,m,\delta')$ are $(3m,\delta')$-free, then $((V_i\mathbf{A}_i V_i^*)_{i=1}^n,(V_i)_{i=1}^n)$ falls in $\Gamma_R(v_1(\varepsilon)\mathbf{X}_1 v_1(\varepsilon)^*\sqcup\cdots\sqcup v_n(\varepsilon)\mathbf{X}_n v_n(\varepsilon)^*,\mathbf{v}(\varepsilon)\,;\,N,m,\delta)$, implying that $(V_i\mathbf{A}_i V_i^*)_{i=1}^n \in \Gamma_R(v_1(\varepsilon)\mathbf{X}_1 v_1(\varepsilon)^*\sqcup\cdots\sqcup v_n(\varepsilon)\mathbf{X}_n v_n(\varepsilon)^*:\mathbf{v}(\varepsilon)\,;\,N,m,\delta)$. We may and do assume that $\delta_{1,\mathrm{orb}}(\mathbf{X}_1,\dots,\mathbf{X}_n) > -\infty$. Thanks to Lemma \ref{L-4-5} there is a subsequence $N_1 < N_2 < \cdots$ such that $\mathbb{P}_{2\sqrt{Cn\varepsilon},R}^\mathrm{orb}(\mathbf{X}_1,\dots,\mathbf{X}_n\,;\,N_k,m,\delta') > -\infty$ for all $k \in \mathbb{N}$ and 
$$
\limsup_{N\to\infty}\frac{1}{N^2}\mathbb{P}_{2\sqrt{Cn\varepsilon},R}^\mathrm{orb}(\mathbf{X}_1,\dots,\mathbf{X}_n\,;\,N,m,\delta') = 
\lim_{k\to\infty}\frac{1}{N_k^2}\mathbb{P}_{2\sqrt{Cn\varepsilon},R}^\mathrm{orb}(\mathbf{X}_1,\dots,\mathbf{X}_n\,;\,N_k,m,\delta'). 
$$
Then for each $k \in \mathbb{N}$ one can choose $\mathbf{A}_i^{(k)} \in ((M_N^\mathrm{sa})_R)^{r(i)}$, $1 \leq i \leq n$, in such a way that 
\begin{align*}
\mathbb{P}_{2\sqrt{Cn\varepsilon},R}^\mathrm{orb}&(\mathbf{X}_1,\dots,\mathbf{X}_n\,;\,N_k,m,\delta') - 1 \\
&< 
\log\Big(P_{2\sqrt{Cn\varepsilon}}\big(\Gamma_\mathrm{orb}(\mathbf{X}_1,\dots,\mathbf{X}_n:(\mathbf{A}_i^{(k)})_{i=1}^n\,;\,N_k,m,\delta')\big)\Big).
\end{align*}
The argument proving \cite[Eq.(5.4)-(5.6)]{HiaiMiyamotoUeda:IJM09} with replacing $N$, $\Xi_i(N)$'s, $\delta$ and $\rho$ there by $N_k$, $\mathbf{A}_i^{(k)}$'s, $\delta'$ and $\delta$, respectively, shows that  
\begin{align*} 
&\frac{1}{2} P_{2\sqrt{Cn\varepsilon}}\big(\Gamma_\mathrm{orb}(\mathbf{X}_1,\dots,\mathbf{X}_n:(\mathbf{A}_i^{(k)})_{i=1}^n\,;\,N_k,m,\delta')\big)\times\gamma_{\mathrm{U}(N_k)}^{\otimes n}\big(\Gamma(\mathbf{v}(\varepsilon)\,;\,N_k,m,\delta')\big) \\
&\leq 
\gamma_{\mathrm{U}(N_k)}^{\otimes n}\big(\Gamma_\mathrm{orb}(v_1(\varepsilon)\mathbf{X}_1 v_1(\varepsilon)^*,\dots,v_n(\varepsilon)\mathbf{X}_n v_n(\varepsilon)^*:(\mathbf{A}_i^{(k)})_{i=1}^n:\mathbf{v}(\varepsilon)\,;\,N_k,m,\delta)\big)
\end{align*} 
for all sufficiently large $k \in \mathbb{N}$. Therefore, we have 
\begin{align*} 
&\limsup_{N\to\infty}\frac{1}{N^2}\bar{\chi}_{\mathrm{orb},R}(v_1(\varepsilon)\mathbf{X}_1 v_1(\varepsilon)^*,\dots,v_n(\varepsilon)\mathbf{X}_n v_n(\varepsilon)^*:\mathbf{v}(\varepsilon)\,;\,N,m,\delta) \\
&\geq 
\limsup_{k\to\infty}\frac{1}{N_k^2}\log\Big(\gamma_{\mathrm{U}(N_k)}^{\otimes n}\big(\Gamma_\mathrm{orb}(v_1(\varepsilon)\mathbf{X}_1 v_1(\varepsilon)^*,\dots,v_n(\varepsilon)\mathbf{X}_n v_n(\varepsilon)^*: \\
&\phantom{aaaaaaaaaaaaaaaaaaaaaaaaaaaaaaaaaaaaaaaaaaaaaai}
(\mathbf{A}_i^{(k)})_{i=1}^n:\mathbf{v}(\varepsilon)\,;\,N_k,m,\delta)\big)\Big) \\
&\geq 
\limsup_{k\to\infty}\frac{1}{N_k^2}\Big[
\log\big(P_{2\sqrt{Cn\varepsilon}}\big(\Gamma_\mathrm{orb}(\mathbf{X}_1,\dots,\mathbf{X}_n:(\mathbf{A}_i^{(k)})_{i=1}^n\,;\,N_k,m,\delta')\big)\big) \\
&\phantom{aaaaaaaaaaaaaaaaaaaaaaaaaaaaaa}+ 
\log\big(\gamma_{\mathrm{U}(N_k)}^{\otimes n}\big(\Gamma(\mathbf{v}(\varepsilon)\,;\,N_k,m,\delta')\big)\big) - \log2\Big] \\
&\geq 
\liminf_{k\to\infty}\frac{1}{N_k^2}\log\big(P_{2\sqrt{Cn\varepsilon}}\big(\Gamma_\mathrm{orb}(\mathbf{X}_1,\dots,\mathbf{X}_n:(\mathbf{A}_i^{(k)})_{i=1}^n\,;\,N_k,m,\delta')\big)\big) \\
&\phantom{aaaaaaaaaaaaaaaaaaa}+ 
\liminf_{k\to\infty}\frac{1}{N_k^2}\log\big(\gamma_{\mathrm{U}(N_k)}^{\otimes n}\big(\Gamma(\mathbf{v}(\varepsilon)\,;\,N_k,m,\delta')\big)\big) \\
&\geq 
\liminf_{k\to\infty}\frac{1}{N_k^2}\Big(\mathbb{P}_{2\sqrt{Cn\varepsilon},R}^\mathrm{orb}(\mathbf{X}_1,\dots,\mathbf{X}_n\,;\,N_k,m,\delta') - 1 \Big) \\
&\phantom{aaaaaaaaaaaaaaaaaaa}+ 
\liminf_{k\to\infty}\frac{1}{N_k^2}\log\big(\gamma_{\mathrm{U}(N_k)}^{\otimes n}\big(\Gamma(\mathbf{v}(\varepsilon)\,;\,N_k,m,\delta')\big)\big) \\
&\geq 
\lim_{k\to\infty}\frac{1}{N_k^2}\mathbb{P}_{2\sqrt{Cn\varepsilon},R}^\mathrm{orb}(\mathbf{X}_1,\dots,\mathbf{X}_n\,;\,N_k,m,\delta') + 
\liminf_{k\to\infty}\frac{1}{N_k^2}\log\big(\gamma_{\mathrm{U}(N_k)}^{\otimes n}\big(\Gamma(\mathbf{v}(\varepsilon)\,;\,N_k,m,\delta')\big)\big) \\
&\geq 
\mathbb{P}_{2\sqrt{Cn\varepsilon}}^\mathrm{orb}(\mathbf{X}_1,\dots,\mathbf{X}_n) + \sum_{i=1}^n \chi_u(v_i(\varepsilon)), 
\end{align*}
where we use Lemma \ref{L-4-5} and \cite[Lemma 5.5]{HiaiMiyamotoUeda:IJM09} in the last line. This implies the desired inequality as in the proof of \cite[Proposition 5.6]{HiaiMiyamotoUeda:IJM09}.  

Let us next prove inequality `$\leq$' in the desired assertion. To do so it suffices to prove that 
$$
\chi_\mathrm{orb}(v_1(\varepsilon)\mathbf{X}_1 v_1(\varepsilon)^*,\dots,v_n(\varepsilon)\mathbf{X}_n v_n(\varepsilon)^*:\mathbf{v}(\varepsilon)) 
\leq 
\mathbb{K}_{\sqrt{\varepsilon}}^\mathrm{orb}(\mathbf{X}_1,\dots,\mathbf{X}_n) + n \log\sqrt{\varepsilon} + \text{Const.}
$$
for every $\varepsilon > 0$. We may and do assume that $\chi_\mathrm{orb}(v_1(\varepsilon)\mathbf{X}_1 v_1(\varepsilon)^*,\dots,v_n(\varepsilon)\mathbf{X}_n v_n(\varepsilon)^*:\mathbf{v}(\varepsilon)) > -\infty$. Let $m \in \mathbb{N}$ and $\delta > 0$ be arbitrary. Thanks to the $\chi_\mathrm{orb}(-:\mathbf{v})$-counterpart of Lemma \ref{L-2-5}  there is a subsequence $N_1 < N_2 < \cdots$ such that $\bar{\chi}_{\mathrm{orb},R}(v_1(\varepsilon)\mathbf{X}_1 v_1(\varepsilon)^*,\dots,v_n(\varepsilon)\mathbf{X}_n v_n(\varepsilon)^*:\mathbf{v}(\varepsilon)\,;\,N_k,3m,\delta) > -\infty$ for all $k \in \mathbb{N}$ and 
\begin{align*} 
&\limsup_{N\to\infty}\frac{1}{N^2}\bar{\chi}_{\mathrm{orb},R}(v_1(\varepsilon)\mathbf{X}_1 v_1(\varepsilon)^*,\dots,v_n(\varepsilon)\mathbf{X}_n v_n(\varepsilon)^*:\mathbf{v}(\varepsilon)\,;\,N,3m,\delta) \\
&= 
\lim_{k\to\infty}\frac{1}{N_k^2}\bar{\chi}_{\mathrm{orb},R}(v_1(\varepsilon)\mathbf{X}_1 v_1(\varepsilon)^*,\dots,v_n(\varepsilon)\mathbf{X}_n v_n(\varepsilon)^*:\mathbf{v}(\varepsilon)\,;\,N_k,3m,\delta). 
\end{align*} 
For each $k \in \mathbb{N}$ one can choose $\mathbf{A}_i^{(k)} \in ((M_N^\mathrm{sa})_R)^{(r(i)}$, $1 \leq i \leq n$, in such a way that 
\begin{align*} 
&\bar{\chi}_{\mathrm{orb},R}(v_1(\varepsilon)\mathbf{X}_1 v_1(\varepsilon)^*,\dots,v_n(\varepsilon)\mathbf{X}_n v_n(\varepsilon)^*:\mathbf{v}(\varepsilon)\,;\,N_k,3m,\delta) - 1 \\
&< 
\log\Big(\gamma_{\mathrm{U}(N_k)}^{\otimes n}\big(\Gamma_\mathrm{orb}(v_1(\varepsilon)\mathbf{X}_1 v_1(\varepsilon)^*,\dots,v_n(\varepsilon)\mathbf{X}_n v_n(\varepsilon)^*:(\mathbf{A}_i^{(k)})_{i=1}^n:\mathbf{v}(\varepsilon)\,;\,N_k,3m,\delta)\big)\Big). 
\end{align*} 
Then the last half of the proof of \cite[Proposition 5.6]{HiaiMiyamotoUeda:IJM09} with replacing $N$ and $\Xi_i(N)$'s there by $N_k$  and $\mathbf{A}_i^{(k)}$'s, respectively, shows that 
\begin{align*} 
&\gamma_{\mathrm{U}(N_k)}^{\otimes n}\big(\Gamma_\mathrm{orb}(v_1(\varepsilon)\mathbf{X}_1 v_1(\varepsilon)^*,\dots,v_n(\varepsilon)\mathbf{X}_n v_n(\varepsilon)^*:(\mathbf{A}_i^{(k)})_{i=1}^n:\mathbf{v}(\varepsilon)\,;\,N_k,3m,\delta)\big) \\
&\leq 
K_{\sqrt{\varepsilon}}\big(\Gamma_\mathrm{orb}(\mathbf{X}_1,\dotsm\mathbf{X}_n:(\mathbf{A}_i^{(k)})_{i=1}^n\,;\,N_k,m,\delta)\big) \times C'((\sqrt{Cn}+1)\sqrt{\varepsilon})^{nN_k^2}
\end{align*} 
for some $C' > 0$, which is independent of $k \in \mathbb{N}$. Therefore, we have 
\begin{align*} 
&\chi_\mathrm{orb}(v_1(\varepsilon)\mathbf{X}_1 v_1(\varepsilon)^*,\dots,v_n(\varepsilon)\mathbf{X}_n v_n(\varepsilon)^*:\mathbf{v}(\varepsilon)) \\
&= 
\chi_{\mathrm{orb},R}(v_1(\varepsilon)\mathbf{X}_1 v_1(\varepsilon)^*,\dots,v_n(\varepsilon)\mathbf{X}_n v_n(\varepsilon)^*:\mathbf{v}(\varepsilon)) \\
&\leq 
\lim_{k\to\infty}\frac{1}{N_k^2}\bar{\chi}_{\mathrm{orb},R}(v_1(\varepsilon)\mathbf{X}_1 v_1(\varepsilon)^*,\dots,v_n(\varepsilon)\mathbf{X}_n v_n(\varepsilon)^*:\mathbf{v}(\varepsilon)\,;\,N_k,3m,\delta) \\
&\leq 
\limsup_{k\to\infty}\frac{1}{N_k^2}
\Big[\log\Big(\gamma_{\mathrm{U}(N_k)}^{\otimes n}\big(\Gamma_\mathrm{orb}(v_1(\varepsilon)\mathbf{X}_1 v_1(\varepsilon)^*,\dots,v_n(\varepsilon)\mathbf{X}_n v_n(\varepsilon)^*:\\
&\phantom{aaaaaaaaaaaaaaaaaaaaaaaaaaaaaaaaaaa}
(\mathbf{A}_i^{(k)})_{i=1}^n:\mathbf{v}(\varepsilon)\,;\,N_k,3m,\delta)\big)\Big)+1\Big] \\
&\leq 
\Big[\limsup_{k\to\infty}\frac{1}{N_k^2}\log\Big(K_{\sqrt{\varepsilon}}\big(\Gamma_\mathrm{orb}(\mathbf{X}_1,\dotsm\mathbf{X}_n:(\mathbf{A}_i^{(k)})_{i=1}^n\,;\,N_k,m,\delta)\big)\Big)\Big] + n\log\sqrt{\varepsilon} + \text{Const.} \\
&\leq 
\Big[\limsup_{k\to\infty}\frac{1}{N_k^2}\mathbb{K}_{\sqrt{\varepsilon},R}^\mathrm{orb}(\mathbf{X}_1,\dotsm\mathbf{X}_n\,;\,N_k,m,\delta)\Big] + n\log\sqrt{\varepsilon} + \text{Const.} \\
&\leq 
\Big[\limsup_{N\to\infty}\frac{1}{N^2}\mathbb{K}_{\sqrt{\varepsilon},R}^\mathrm{orb}(\mathbf{X}_1,\dotsm\mathbf{X}_n\,;\,N,m,\delta)\Big] + n\log\sqrt{\varepsilon} + \text{Const.}, 
\end{align*}
where the first equality is due to the $\chi_\mathrm{orb}(-:\mathbf{v})$-counterpart of Lemma \ref{L-2-5}. This together with Lemma \ref{L-4-5} immediately implies the desired inequality. 
\end{proof} 

As mentioned in the introduction it was shown in \cite[Theorem 5.8]{HiaiMiyamotoUeda:IJM09} that
$$
\delta_0(\mathbf{X}_1,\dots,\mathbf{X}_n) = 
\delta_{0,\mathrm{orb}}(\mathbf{X}_1,\dots,\mathbf{X}_n) + \sum_{i=1}^n \delta_0(\mathbf{X}_i) 
$$
holds for arbitrary hyperfinite random multi-variables $\mathbf{X}_1,\dots,\mathbf{X}_n$. 
It would be nice if one could prove even inequality `$\leq$' in the above without the hyperfiniteness assumption. However, it seems difficult at the moment to do so; one reason is that the proof of \cite[Theorem 5.8]{HiaiMiyamotoUeda:IJM09} heavily depends upon the thorough investigation on $\delta_0(\mathbf{X})$ when $W^*(\mathbf{X})$ is hyperfinite due to Jung \cite{Jung:TAMS03}. 

\section*{Appendix: Talagrand's Inequality for $\chi_\mathrm{orb}$} 

Set $R := \max\{\Vert X_{ij}\Vert_\infty\,|\,1 \leq j \leq r(i), 1 \leq i \leq n\}$. Let $\mathcal{A}_R$ be the universal free product $C^*$-algebra of $r(1)+\cdots+r(n)$-copies of $C[-R,R]$. Then we can define the tracial states $\tau_{(\mathbf{X}_1,\dots,\mathbf{X}_n)}$ and $\tau^\mathrm{free}_{(\mathbf{X}_1,\dots,\mathbf{X}_n)}$ on $\mathcal{A}_R$ as in the proof of \cite[Proposition 4.4 (8)]{HiaiMiyamotoUeda:IJM09}. We may and do assume that $\chi_\mathrm{orb}(\mathbf{X}_1,\dots,\mathbf{X}_n) > -\infty$. Then one can choose a subsequence $N_1 < N_2 < \cdots$ in such a way that $\bar{\chi}_{\mathrm{orb},R}(\mathbf{X}_1,\dots,\mathbf{X}_n\,;\,N_m,m,1/m) > -\infty$ for every $m \in \mathbb{N}$ and 
$$
\chi_\mathrm{orb}(\mathbf{X}_1,\dots,\mathbf{X}_n) = 
\lim_{m\to\infty}\frac{1}{N_m^2}\bar{\chi}_{\mathrm{orb},R}(\mathbf{X}_1,\dots,\mathbf{X}_n\,;\,N_m,m,1/m),  
$$
where Corollary \ref{C-2-7} is used. Remark that it makes no change to replace $\mathrm{U}(N)$ by $\mathrm{SU}(N)$ in Proposition \ref{P-2-4}. (In fact, this follows from that the inverse image of $\Gamma_\mathrm{orb}(\mathbf{X}_1,\dots,\mathbf{X}_n:(\mathbf{A}_i)_{i=1}^n\,;\,N,m,\delta)$ under the surjection $(\zeta,V) \in \mathbb{T}\times\mathrm{SU}(N) \mapsto \zeta V \in \mathrm{U}(N)$ is the product of $\mathbb{T}$ and $\mathrm{SU}(N) \cap \Gamma_\mathrm{orb}(\mathbf{X}_1,\dots,\mathbf{X}_n:(\mathbf{A}_i)_{i=1}^n\,;\,N,m,\delta)$ and that the push-forward of the Haar probability measure of $\mathbb{T}\times\mathrm{SU}(N)$ under the surjection is $\gamma_{\mathrm{U}(N)}$.) Thus, for each $m \in \mathbf{N}$ one can choose $\mathbf{A}_i^{(m)} \in ((M_N^\mathrm{sa})_R)^{r(i)}$ in such a way that 
$$
\bar{\chi}_{\mathrm{orb},R}(\mathbf{X}_1,\dots,\mathbf{X}_n\,;\,N_m,m,1/m) - 1 < 
\log\Big(\gamma_{\mathrm{SU}(N_m)}^{\otimes n}\big(\Gamma_m\big)\Big), 
$$ 
with $\Gamma_m := \mathrm{SU}(N_m) \cap \Gamma_\mathrm{orb}(\mathbf{X}_1,\dots,\mathbf{X}_n:(\mathbf{A}_i^{(m)})_{i=1}^n\,;\,N_m,m,1/m)$. The same argument as in \cite[\S3 and Proposition 4.4 (8)]{HiaiMiyamotoUeda:IJM09} with replacing $\Xi_i(N_m)$'s there by 
$\mathbf{A}_i^{(m)}$'s shows that 
\begin{align*} 
W_{2,\mathrm{free}}&(\tau_{(\mathbf{X}_1,\dots,\mathbf{X}_n)},\tau^\mathrm{free}_{(\mathbf{X}_1,\dots,\mathbf{X}_n)}) 
\leq 
4R\sqrt{r}\liminf_{m\to\infty}\sqrt{-\frac{1}{N_m^2}\log\Big(\gamma_{\mathrm{SU}(N_m)}^{\otimes n}\big(\Gamma_m\big)\Big)} \\
&\leq 4R\sqrt{r}\sqrt{-\limsup_{m\to\infty}\frac{1}{N_m^2}\big(\bar{\chi}_{\mathrm{orb},R}(\mathbf{X}_1,\dots,\mathbf{X}_n\,;\,N_m,m,1/m)-1\big)} \\
&= 4R\sqrt{r}\sqrt{-\chi_\mathrm{orb}(\mathbf{X}_1,\dots,\mathbf{X}_n)}  
\end{align*} 
with $r := \max\{r(i)\,|\,1\leq i \leq n\}$. At this point we thank the referee who made us aware  that the constant $\sqrt{r}$ above is missing in the previous version of this paper and also in the proof of \cite[Proposition 4.4 (8)]{HiaiMiyamotoUeda:IJM09}. The constant $\sqrt{r}$ appears in the estimate comparing the free Wasserstein distance for tracial states on $\mathcal{A}_R$ and the ordinary one for measures on $SU(N)^n$ (see the proof of \cite[Lemma 3.4]{HiaiMiyamotoUeda:IJM09} which deals with only the case when all $r(i)=1$, that is, $r=1$ in the case).   

\section*{Added in proof} 

A preprint version of \cite{CollinsKemp:JFA14} due to Collins and Kemp appeared after the release of the present paper as a preprint. Their paper focuses in part on giving a proof of $i^*=-\chi_\mathrm{orb}$ for two projections of special kind along the lines of our previous heuristic argument in \cite{HiaiUeda:AIHP09}. Our proof (or approach) mentioned in the introduction is certainly different from and much simpler than theirs, and it is now available with stronger results in \cite{IzumiUeda:Preprint13}. We have two different approaches now, and it is probably interesting to compare those thoroughly.   

\section*{Acknowledgments} We thank Professor Fumio Hiai for some comments to a draft of this paper. We are also very grateful to the anonymous referee for reading the paper very carefully and giving detailed comments. 
}

\end{document}